\documentclass[11pt]{article}

\usepackage[a4paper,margin=1in]{geometry}
\usepackage[T1]{fontenc}
\usepackage[utf8]{inputenc}
\usepackage[english]{babel}

\usepackage{amsmath,amssymb,amsthm,mathtools}
\usepackage{graphicx}
\usepackage{wrapfig}
\usepackage{tikz}
\usepackage[hidelinks]{hyperref}
\usepackage{cleveref}

\newtheorem{theorem}{Theorem}

\newtheorem{lemma}{Lemma}

\theoremstyle{definition}

\newtheorem{remark}{Remark}

\title{
  Area of the Limit Curve\\
  for Symmetric Nonstationary\\
  Corner-Cutting Schemes
}

\author{
  Georgii Khoruzhevskii\\
  Independent researcher, Belgrade, Serbia\\
  \texttt{khoruzhevskiigeorgii@gmail.com}
}

\date{}

\begin{document}

\maketitle

\begin{abstract}
We consider corner-cutting schemes for closed oriented polygons in the plane.
At each step every vertex is replaced by two points on the adjacent edges; the
same cutting coefficient is used at all vertices of the current polygon, but
this coefficient may vary from step to step. For such schemes we prove the
existence of a continuous limit curve and derive an explicit formula for its
oriented area. The main result separates two contributions: the geometry of
the initial contour enters through its oriented area and a local sum of
determinants built from neighboring edges, while the entire dependence on the
refinement rule is contained in a single numerical constant. We also discuss
the geometric meaning of this constant as an area arising in the standard
angle. As applications we consider stationary de Rham curves, the Chaikin
scheme, the connection with Archimedes' quadrature of the parabola, and special
sequences of coefficients leading to a circle and to an arc of a hyperbola;
these give corresponding series for \(\pi\) and \(\log 2\).
\end{abstract}

\tableofcontents

\section{Introduction}

Consider a closed polygonal contour: a rough sketch of a future curve. To make
it smoother, one may cut every corner. On the two edges meeting at a vertex, two
new points are chosen, and the vertex itself is replaced by the short segment
between them. After one step one obtains a new polygon with more sides; after
many steps the contour looks increasingly smooth.

\begin{figure}[ht]
  \centering
  \begin{tikzpicture}[
      x=1cm,
      y=1cm,
      line cap=round,
      vertex/.style={circle,fill=black,inner sep=1.4pt},
      cutpoint/.style={circle,fill=white,draw=black,inner sep=1.3pt},
      original/.style={line width=1.1pt,black},
      first/.style={line width=1.1pt,blue!70!black},
      second/.style={line width=1.1pt,red!70!black},
      continuation/.style={line width=0.9pt,densely dashed,opacity=0.55},
      label/.style={font=\small,anchor=east,align=right},
      row/.style={xshift=3.3cm}
    ]
    \node[label] at (2.50,0.70) {initial\\polyline};
    \begin{scope}[row]
      \draw[original,continuation] (-0.55,-0.46) -- (0,0);
      \draw[original] (0,0) -- (1.6,1.35) -- (3.4,0.15) -- (5.0,1.25);
      \draw[original,continuation] (5.0,1.25) -- (5.55,1.63);
      \foreach \x/\y in {0/0,1.6/1.35,3.4/0.15,5.0/1.25} \node[vertex] at (\x,\y) {};
    \end{scope}

    \node[label] at (2.50,-1.45) {after one\\step};
    \begin{scope}[row,yshift=-2.15cm]
      \draw[black!25,line width=0.8pt] (0,0) -- (1.6,1.35) -- (3.4,0.15) -- (5.0,1.25);
      \draw[first,continuation] (0.25,0.21) -- (0.70,0.59);
      \draw[first] (0.70,0.59) -- (1.15,0.97) -- (2.10,1.01) -- (2.90,0.49) -- (3.85,0.46) -- (4.30,0.77);
      \draw[first,continuation] (4.30,0.77) -- (4.75,1.08);
      \foreach \x/\y in {1.15/0.97,2.10/1.01,2.90/0.49,3.85/0.46} \node[cutpoint] at (\x,\y) {};
    \end{scope}

    \node[label] at (2.50,-3.60) {after two\\steps};
    \begin{scope}[row,yshift=-4.30cm]
      \draw[black!20,line width=0.8pt] (0,0) -- (1.6,1.35) -- (3.4,0.15) -- (5.0,1.25);
      \draw[blue!25,line width=0.8pt] (1.15,0.97) -- (2.10,1.01) -- (2.90,0.49) -- (3.85,0.46);
      \draw[second,continuation] (0.38,0.32) -- (0.60,0.51);
      \draw[second] (0.60,0.51) -- (0.83,0.70) -- (1.42,0.98) -- (1.84,1.00) -- (2.33,0.87) -- (2.67,0.63) -- (3.16,0.48) -- (3.58,0.47) -- (4.17,0.68) -- (4.39,0.83);
      \draw[second,continuation] (4.39,0.83) -- (4.62,0.99);
      \foreach \x/\y in {0.83/0.70,1.42/0.98,1.84/1.00,2.33/0.87,2.67/0.63,3.16/0.48,3.58/0.47,4.17/0.68} \node[cutpoint] at (\x,\y) {};
    \end{scope}
  \end{tikzpicture}
  \caption{A local picture of corner cutting on a fragment with three edges.}
  \label{fig:local-corner-cutting}
\end{figure}
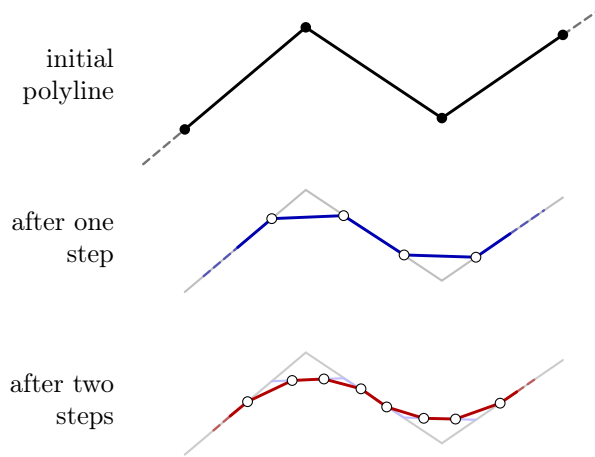

A natural question arises: can one find the area of the limit curve without
being able to write the curve itself explicitly? It turns out that, for the
schemes considered here, this is possible. Moreover, the entire geometry of the
initial polygon is described by only two quantities, while the infinite sequence
of cuts contributes a single universal constant. At the end of the paper we
will see how this leads to geometric series for \(\pi\) and \(\log 2\).

This is precisely how corner-cutting schemes work. They are a classical way to
construct limit curves from polygonal control data; in modern terminology this
is one of the simplest examples of subdivision schemes \cite{DynLevin2002}. The
best-known example is Chaikin's algorithm, whose limit is a quadratic
\(B\)-spline\begingroup\renewcommand{\thefootnote}{\fnsymbol{footnote}}%
\footnote[1]{A spline is a curve glued from simple polynomial pieces. A
quadratic spline means that each such piece is given by polynomials of degree at
most two; geometrically it is an arc of a parabola or a degenerate case of one.
A \(B\)-spline is not a new kind of curve, but a convenient way to write such
smoothly glued pieces in terms of control points; the letter \(B\) stands for
\emph{basis}. Thus in this note the phrase ``quadratic \(B\)-spline'' may be
read as ``a smoothly glued chain of parabolic pieces''.}\endgroup{} curve, that
is, roughly speaking, a smoothly glued chain of parabolic pieces.

In this note we consider a particularly simple but flexible variant: at each
refinement level the same coefficient is used for all corners. When passing to
the next level, this coefficient is allowed to change. This choice keeps the
same treatment of all vertices within a single step while allowing the
refinement rule to vary from level to level.

We are interested not so much in the smoothness of the limit curve as in its
area. Even if the curve itself is hard to write explicitly, its area can be
found directly from the initial polygon and the cutting coefficients. We use
oriented area: it records the direction in which the contour is traversed, and
therefore changes sign when the orientation is reversed.

The main result may be read as follows: the area of the limit curve is the area
of the initial polygon plus a very simple correction. If
\(P=(p_0,\ldots,p_{n-1})\) is the initial oriented polygon and
\((c_k)_{k\ge 0}\) is the sequence of cutting coefficients, then the oriented
area of the limit curve has the form
\[
  A_\infty
  =
  A(P)
  +
  C(c_\bullet)
  \sum_i
  \det(p_{i-1}-p_i,\ p_{i+1}-p_i),
\]
Here \(A(P)\) is the oriented area of the initial polygon. The sum of
determinants describes the local turns of the initial contour near its vertices,
whereas the number \(C(c_\bullet)\) depends only on the chosen sequence of
cuts. This constant has the form
\[
  C(c_\bullet)
  =
  \sum_{k=0}^{\infty}
  2^k\frac{c_k^2}{2}
  \prod_{j=0}^{k-1}c_j(1-2c_j),
\]
where the empty product for \(k=0\) is equal to \(1\). In other words, the
geometry of the initial polygon and the smoothing rules enter the formula
separately.

To read the paper it is enough to know the determinant of two vectors, the area
formula for a polygon, and elementary properties of series. In some examples we
also use a parametrization of a curve and very simple integrals.

\section{The Corner-Cutting Scheme}

Let us pass from the picture to precise notation. Below, a polygon means an
ordered cyclic collection of points in the plane. Thus the notation
\[
    P=(p_0,p_1,\ldots,p_{n-1})
\]
means that the vertices are connected in the indicated order, and after
\(p_{n-1}\) one returns to \(p_0\). The polygon need not be convex; for the area
formula, the orientation of the traversal is what matters.

All indices below are understood cyclically. If a polygon has \(N\) vertices,
then \(p_N=p_0\), \(p_{-1}=p_{N-1}\), and so on. This convention lets us avoid
separate endpoint cases.

Fix a sequence of cutting coefficients
\[
    c_0,c_1,c_2,\ldots,
    \qquad
    0<c_m<\frac12.
\]
The coefficient \(c_m\) is used at the \(m\)-th step for all vertices
simultaneously. In other words, within one step the rule is the same for all
vertices, but it may change from one step to the next.

Let us describe one step. Suppose an oriented polygon has already been
constructed:
\[
    P^{(m)}=\bigl(p^{(m)}_0,p^{(m)}_1,\ldots,p^{(m)}_{N_m-1}\bigr).
\]
On each of its edges
\([p^{(m)}_i,p^{(m)}_{i+1}]
\) we take two points
\[
    p^{(m+1)}_{2i}
    =
    (1-c_m)p^{(m)}_i+c_m p^{(m)}_{i+1},
\]
\[
    p^{(m+1)}_{2i+1}
    =
    c_m p^{(m)}_i+(1-c_m)p^{(m)}_{i+1}.
\]
The first point lies closer to the beginning of the old edge, and the second
closer to its end. The old vertices are then removed, and all new points are
connected in the order of their indices. The resulting polygon is denoted by
\(P^{(m+1)}\).

Equivalently, each old vertex \(p^{(m)}_i\) is replaced by a short new edge
connecting two points on the adjacent old edges. This edge is precisely the
``cut-off corner''. Each old edge gives two new vertices, hence
\[
    N_m=n2^m.
\]

The condition \(0<c_m<1/2\) means that the cut is nondegenerate. If \(c_m=0\),
the shape of the polygon would in fact not change, while if \(c_m=1/2\), the
two new points on each old edge would coincide at its midpoint. These limiting
cases are useful for intuition, but we do not include them in the main scheme.

For an arbitrary \(c\in(0,1/2)\), denote by \(T_c\) the one-step operator given
by the rule above. Then the entire sequence of polygons is defined recursively:
\[
    P^{(m+1)}=T_{c_m}P^{(m)},
    \qquad
    P^{(0)}=P.
\]
Equivalently,
\[
    P^{(m)}
    =
    T_{c_{m-1}}\circ T_{c_{m-2}}\circ\cdots\circ T_{c_0}(P),
    \qquad m\ge 1.
\]

Below we check that infinite iteration of this procedure produces a continuous
limit curve. We denote it by \(P^{(\infty)}\). We shall be interested both in
its parametric equation and in the oriented area that it bounds.

\section{Oriented Area of Polygons}

In this section we fix notation for area. For two vectors \(u=(u_1,u_2)\) and
\(v=(v_1,v_2)\), write
\[
    [u,v]
    =
    \det
    \begin{pmatrix}
        u_1 & u_2 \\
        v_1 & v_2
    \end{pmatrix}
    =
    u_1v_2-u_2v_1.
\]
Geometrically, \([u,v]\) is the oriented area of the parallelogram spanned by
\(u\) and \(v\).

In this section the letter \(Q\) denotes an arbitrary oriented polygon. This is
convenient because it distinguishes it from the initial polygon \(P\): later the
formulas obtained here will be applied to current polygons \(Q=P^{(m)}\).

Thus let
\[
    Q=(q_0,\ldots,q_{N-1})
\]
be an oriented polygon in the plane. Its oriented area is defined by
\[
    A(Q)
    =
    \frac12
    \sum_{i=0}^{N-1}
    [q_i,q_{i+1}],
\]
where, as usual, \(q_N=q_0\).

This formula is known as Gauss' area formula, or the shoelace
formula\begingroup\renewcommand{\thefootnote}{\fnsymbol{footnote}}%
\footnote[1]{The formula is a special case of Green's theorem, which in turn is
a special case of Stokes' theorem. For readers not familiar with them, induction
is always available.}\endgroup. If the coordinates of the vertices are
\(q_i=(x_i,y_i)\), it becomes
\[
    A(Q)
    =
    \frac12
    \sum_{i=0}^{N-1}
    (x_i y_{i+1}-y_i x_{i+1}).
\]

It is important that \(A(Q)\) is an oriented area. If the vertices of an ordinary
convex polygon are traversed counterclockwise, it is positive. If the direction
of traversal is reversed, the sign of the area changes:
\[
    A(q_{N-1},q_{N-2},\ldots,q_0)=-A(Q).
\]
For self-intersecting polygons this quantity also makes sense, now as an
algebraic area: regions traversed with different orientations are counted with
different signs.

We shall also need a local quantity associated with a single vertex of a
polygon. For the vertex \(q_i\), consider the two vectors issuing from it toward
the neighboring vertices:
\[
    q_{i-1}-q_i,
    \qquad
    q_{i+1}-q_i.
\]
Their determinant
\[
    [q_{i-1}-q_i,\;q_{i+1}-q_i]
\]
is twice the oriented area of the triangle with vertices
\(q_{i-1},q_i,q_{i+1}\). For convenience, introduce also
\[
    B(Q)
    =
    \sum_{i=0}^{N-1}
    [q_{i-1}-q_i,\;q_{i+1}-q_i].
\]

This collects, over all vertices of the polygon, the determinants built from the
two edges adjacent to the given vertex. Geometrically each summand is twice the
oriented area of the triangle with vertices \(q_{i-1},q_i,q_{i+1}\). This
interpretation will not be essential for us, however: one may simply regard
\(B(Q)\) as an abbreviation that makes the subsequent computations more
readable.

It is precisely the quantity \(B(Q)\) that will appear in the formula for the
change of area under one corner-cutting step. In particular, for the initial
polygon \(P=(p_0,\ldots,p_{n-1})\), we shall use the notation
\[
    B(P)
    =
    \sum_{i=0}^{n-1}
    [p_{i-1}-p_i,\;p_{i+1}-p_i].
\]

\section{The Limit Curve and Convergence of Areas}

Let us explain why the computations below are legitimate: namely, why the limit
curve actually exists and why its area can be obtained as the limit of the areas
of the approximating polygons.

Denote by
\[
    s_i=\frac{p_i+p_{i+1}}2
\]
the midpoint of the edge \([p_i,p_{i+1}]\) of the initial polygon. These points
are convenient as boundaries of local pieces: the part of the curve between
\(s_{i-1}\) and \(s_i\) is constructed only from the three neighboring vertices
\(p_{i-1},p_i,p_{i+1}\). This follows from the locality of the corner-cutting
scheme.

\begin{lemma}
The sequence of polygons \(P^{(m)}\) has a continuous closed limit curve
\(P^{(\infty)}\) in the Hausdorff metric. Moreover, if
\[
    A_m=A(P^{(m)}),
\]
then the oriented area of the limit curve is the limit of the oriented areas of
the approximating polygons:
\[
    A_\infty=\lim_{m\to\infty}A_m.
\]
\end{lemma}

\begin{proof}
We prove the assertion locally, on the piece between \(s_{i-1}\) and \(s_i\).
If the points \(p_{i-1},p_i,p_{i+1}\) lie on one line, then all subsequent local
pieces also lie on this line, and the assertion is obvious. We therefore assume
that the corresponding triangle is nondegenerate.

Make an affine change of coordinates sending the segment \([s_{i-1},s_i]\) to
the interval \([0,1]\) on the \(x\)-axis and the vertex \(p_i\) to a point in
the upper half-plane. After this change each local piece of \(P^{(m)}\) is the
graph of a continuous piecewise-linear function
\(f_{i,m}\colon[0,1]\to\mathbb R\). Passing from level \(m\) to level
\(m+1\) only cuts corners of the previous graph. Hence
\[
    0\leq f_{i,m+1}(x)\leq f_{i,m}(x)\leq f_{i,0}(x)
    \qquad x\in[0,1].
\]
Therefore, for each \(x\), the limit
\[
    f_i(x)=\lim_{m\to\infty}f_{i,m}(x)
\]
exists.

\begin{figure}[ht]
  \centering
  \begin{tikzpicture}[
      x=5.6cm,
      y=2.3cm,
      line cap=round,
      axis/.style={->,line width=0.7pt,black!65},
      base/.style={line width=0.7pt,black!30},
      fzero/.style={line width=1.0pt,black!60},
      fone/.style={line width=1.1pt,blue!65!black},
      ftwo/.style={line width=1.1pt,red!65!black},
      flimit/.style={line width=1.2pt,green!45!black,dashed},
      point/.style={circle,fill=black,inner sep=1.1pt},
      label/.style={font=\small}
    ]
    \draw[axis] (-0.04,0) -- (1.08,0) node[right,label] {\(x\)};
    \draw[axis] (0,-0.06) -- (0,1.08) node[above,label] {\(y\)};
    \draw[base] (0,0) -- (1,0);

    \draw[fzero]
      (0,0) -- (0.5,1.0) -- (1,0);
    \draw[fone]
      (0,0) -- (0.25,0.5) -- (0.75,0.5) -- (1,0);
    \draw[ftwo]
      (0,0) -- (0.125,0.25) -- (0.375,0.5) -- (0.625,0.5)
      -- (0.875,0.25) -- (1,0);
    \draw[flimit]
      plot[domain=0:1,samples=80] (\x,{2*\x*(1-\x)});

    \node[point,label={[label]below left:\(s_{i-1}\)}] at (0,0) {};
    \node[point,label={[label]below:\(s_i\)}] at (1,0) {};
    \node[label] at (0.51,1.04) {\(f_{i,0}\)};
    \node[label,blue!65!black] at (0.80,0.54) {\(f_{i,1}\)};
    \node[label,red!65!black] at (0.64,0.62) {\(f_{i,2}\)};
    \node[label,green!45!black] at (0.50,0.36) {\(f_i\)};

    \draw[->,black!50,line width=0.7pt] (0.30,0.82) -- (0.30,0.56);
    \draw[->,black!50,line width=0.7pt] (0.30,0.49) -- (0.30,0.37);
    \node[label,align=center,black!60] at (0.16,0.67)
      {\(m\) grows};
  \end{tikzpicture}
  \caption{Local convergence of graphs for the example \(c_m\equiv 1/4\):
  \(f_{i,0}\) is the initial polyline, \(f_{i,1}\) and \(f_{i,2}\) are the first
  two corner-cutting steps, and the dashed curve is the limit graph.}
  \label{fig:local-graph-convergence}
\end{figure}
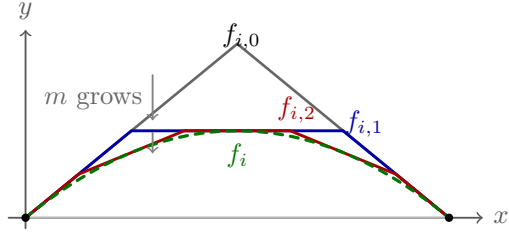

Moreover, the functions \(f_{i,m}\) are concave in the same sense as the
initial polyline on this triangle: each next function is obtained from the
previous one by finitely many cuts by straight lines. Hence the limit function
\(f_i\) is concave as well. Such a function is continuous in the interior of the
interval \cite{BoltyanskyYaglom1951}, and continuity at the endpoints follows
from the estimate
\[
    0\leq f_i(x)\leq f_{i,0}(x),
\]
since \(f_{i,0}(0)=f_{i,0}(1)=0\). Now the functions \(f_{i,m}\) are
continuous on the compact interval \([0,1]\) and decrease monotonically to the
continuous function \(f_i\). Hence, by Dini's theorem, the convergence is
uniform. Uniform convergence implies convergence of the graphs in the
Hausdorff metric. Thus each local limit piece is a continuous curve.
Neighboring pieces have common endpoints \(s_i\), and therefore finitely many
such pieces glue together to a continuous closed curve \(P^{(\infty)}\).

It remains to check convergence of areas. In the chosen affine coordinates, the
oriented area of the local piece closed by the segment \([s_{i-1},s_i]\) differs
from the integral
\[
    \int_0^1 f_{i,m}(x)\,dx
\]
only by a constant oriented factor, namely the determinant of the inverse affine
transformation. At the same time,
\[
    |f_{i,m}(x)|\leq f_{i,0}(x),
\]
and the function \(f_{i,0}\) is integrable. Therefore, by Lebesgue's dominated
convergence theorem,
\[
    \int_0^1 f_{i,m}(x)\,dx
    \longrightarrow
    \int_0^1 f_i(x)\,dx.
\]

The area of the whole polygon decomposes into the area of the fixed polygon
with vertices \(s_0,\ldots,s_{n-1}\) and finitely many such local contributions.
The limit curve has the same decomposition, only with the functions \(f_i\).
Since there are finitely many local pieces, we may pass to the limit in each of
them and sum the results. Consequently,
\[
    A(P^{(m)})\longrightarrow A(P^{(\infty)}).
\]
We denote this limiting oriented area by \(A_\infty\).
\end{proof}

The next simple lemma will be used in the examples, when we compare a known
smooth curve with the limit curve obtained by corner cutting.

\begin{lemma}[Comparison on a dense set]\label{lem:dense-comparison}
Let \(\gamma,\widetilde\gamma\colon[a,b]\to\mathbb R^2\) be two continuous
curves. If they coincide on an everywhere dense set \(D\subset[a,b]\), then
they coincide on the whole interval.

One convenient way to obtain such a set \(D\) is the following. Let \(\Pi_m\) be
a nested sequence of partitions of the interval \([a,b]\), and let \(|\Pi_m|\)
be the length of the largest subinterval of the partition. If
\[
    |\Pi_m|\longrightarrow 0,
\]
then the set of all points of all partitions
\[
    D=\bigcup_{m\geq0}\Pi_m
\]
is everywhere dense in \([a,b]\). In particular, it is enough that, when passing
from \(\Pi_m\) to \(\Pi_{m+1}\), each old subinterval is replaced by
subintervals of length at most \(\lambda_m\) times its length, where
\(0<\lambda_m<1\) and
\[
    \prod_{j=0}^{m-1}\lambda_j\longrightarrow0.
\]
For example, this is certainly satisfied if each subinterval is divided by one
point in the ratio \(\theta_m:(1-\theta_m)\) and all \(\theta_m\) stay away from
\(0\) and \(1\):
\[
    \eta\leq\theta_m\leq1-\eta
    \qquad\text{for some }\eta>0.
\]
\end{lemma}

\begin{proof}
We take as standard the fact that continuous maps coinciding on an everywhere
dense set coincide everywhere. It remains only to check that the set of
partition points is dense. Take any open interval \((u,v)\subset[a,b]\). If
\(|\Pi_m|<v-u\), then \((u,v)\) must contain a point of the partition \(\Pi_m\);
otherwise the entire interval \((u,v)\) would lie inside one subinterval of the
partition of length at least \(v-u\), contradicting the choice of \(m\). Hence
every open interval meets \(D\), so \(D\) is everywhere dense.

Finally, if every new subinterval has length at most \(\lambda_m\) times the
length of the old one, then
\[
    |\Pi_m|\leq
    (b-a)\prod_{j=0}^{m-1}\lambda_j.
\]
Thus under the stated condition \(|\Pi_m|\to0\).
\end{proof}

\section{The One-Step Area Increment}

We now compute how the oriented area changes under one corner-cutting step.

Let
\[
    Q=(q_0,\ldots,q_{N-1})
\]
be an oriented polygon, and let \(T_cQ\) be the polygon obtained from it by
corner cutting with coefficient \(c\), where
\[
    0<c<\frac12.
\]
Denote the vertices of the new polygon by
\[
    r_0,\ldots,r_{2N-1}.
\]
By definition of the scheme,
\[
    r_{2i}=(1-c)q_i+cq_{i+1},
\]
\[
    r_{2i+1}=cq_i+(1-c)q_{i+1}.
\]

\begin{lemma}
For every oriented polygon \(Q\),
\[
    A(T_cQ)
    =
    A(Q)+\frac{c^2}{2}B(Q).
\]
\end{lemma}

\begin{proof}
By the oriented area formula,
\[
    2A(T_cQ)
    =
    \sum_{i=0}^{N-1}
    \bigl([r_{2i},r_{2i+1}]+[r_{2i+1},r_{2i+2}]\bigr).
\]
Substitute the expressions for the new vertices.

First,
\[
    [r_{2i},r_{2i+1}]
    =
    [(1-c)q_i+cq_{i+1},\;cq_i+(1-c)q_{i+1}].
\]
By bilinearity and skew-symmetry of the determinant,
\[
    [r_{2i},r_{2i+1}]
    =
    (1-2c)[q_i,q_{i+1}].
\]

Next,
\[
    [r_{2i+1},r_{2i+2}]
    =
    [cq_i+(1-c)q_{i+1},\;(1-c)q_{i+1}+cq_{i+2}].
\]
Expanding the determinant gives
\[
    [r_{2i+1},r_{2i+2}]
    =
    c(1-c)[q_i,q_{i+1}]
    +c^2[q_i,q_{i+2}]
    +c(1-c)[q_{i+1},q_{i+2}].
\]

Now sum over all \(i\). After an index shift, the last sum coincides with
\(\sum_i [q_i,q_{i+1}]\). Therefore
\[
    2A(T_cQ)
    =
    (1-2c^2)\sum_{i=0}^{N-1}[q_i,q_{i+1}]
    +
    c^2\sum_{i=0}^{N-1}[q_i,q_{i+2}].
\]
Since
\[
    2A(Q)=\sum_{i=0}^{N-1}[q_i,q_{i+1}],
\]
we obtain
\[
    2A(T_cQ)-2A(Q)
    =
    c^2
    \left(
        \sum_{i=0}^{N-1}[q_i,q_{i+2}]
        -
        2\sum_{i=0}^{N-1}[q_i,q_{i+1}]
    \right).
\]

It remains to observe that the expression in parentheses is \(B(Q)\). Indeed,
\[
    B(Q)
    =
    \sum_{i=0}^{N-1}
    [q_{i-1}-q_i,\;q_{i+1}-q_i].
\]
Expanding the determinant, we get
\[
    [q_{i-1}-q_i,\;q_{i+1}-q_i]
    =
    [q_{i-1},q_{i+1}]
    -
    [q_{i-1},q_i]
    -
    [q_i,q_{i+1}].
\]
After summation and an index shift this becomes
\[
    B(Q)
    =
    \sum_{i=0}^{N-1}[q_i,q_{i+2}]
    -
    2\sum_{i=0}^{N-1}[q_i,q_{i+1}].
\]
Consequently,
\[
    2A(T_cQ)-2A(Q)=c^2B(Q),
\]
that is,
\[
    A(T_cQ)
    =
    A(Q)+\frac{c^2}{2}B(Q).
\]
\end{proof}

The formula obtained has a simple geometric meaning. Under corner cutting, the
area changes by the sum of the areas of the small triangles cut off near the old
vertices. The sign of this correction depends on the orientation of the
traversal. For example, for a convex polygon traversed counterclockwise, the
quantity \(B(Q)\) is negative, and therefore the area decreases.

For later use we need one more identity: how the quantity \(B(Q)\) itself
changes under corner cutting.

\begin{lemma}
For every oriented polygon \(Q\),
\[
    B(T_cQ)=2c(1-2c)B(Q).
\]
\end{lemma}

\begin{proof}
Let, as above,
\[
    r_{2i}=(1-c)q_i+cq_{i+1},
    \qquad
    r_{2i+1}=cq_i+(1-c)q_{i+1}.
\]
The sum \(B(T_cQ)\) contains terms of two types: those around vertices
\(r_{2i}\) and those around vertices \(r_{2i+1}\).

For the vertex \(r_{2i}\), we have
\[
    r_{2i-1}-r_{2i}
    =
    c(q_{i-1}-q_{i+1}),
\]
\[
    r_{2i+1}-r_{2i}
    =
    (1-2c)(q_{i+1}-q_i).
\]
Hence the corresponding summand is
\[
    c(1-2c)
    [q_{i-1}-q_{i+1},\;q_{i+1}-q_i].
\]

For the vertex \(r_{2i+1}\), similarly,
\[
    r_{2i}-r_{2i+1}
    =
    (1-2c)(q_i-q_{i+1}),
\]
\[
    r_{2i+2}-r_{2i+1}
    =
    c(q_{i+2}-q_i).
\]
Therefore the second summand is
\[
    c(1-2c)
    [q_i-q_{i+1},\;q_{i+2}-q_i].
\]

Summing over all \(i\), we obtain
\[
    B(T_cQ)
    =
    c(1-2c)
    \sum_i
    [q_{i-1}-q_{i+1},\;q_{i+1}-q_i]
\]
\[
    +
    c(1-2c)
    \sum_i
    [q_i-q_{i+1},\;q_{i+2}-q_i].
\]
The first sum is equal to \(B(Q)\), because
\[
    q_{i-1}-q_{i+1}
    =
    (q_{i-1}-q_i)+(q_i-q_{i+1}),
\]
and the additional term gives a zero determinant.

The second sum is also equal to \(B(Q)\) after an index shift. Therefore
\[
    B(T_cQ)=c(1-2c)B(Q)+c(1-2c)B(Q),
\]
that is,
\[
    B(T_cQ)=2c(1-2c)B(Q).
\]
\end{proof}

Thus one corner-cutting step gives two simple recurrence formulas:
\[
    A(T_cQ)=A(Q)+\frac{c^2}{2}B(Q),
\]
\[
    B(T_cQ)=2c(1-2c)B(Q).
\]
Everything is now ready for the main result.

\section{The Main Area Formula}

We now apply the formulas of the previous section to the whole sequence of
polygons
\[
    P^{(0)},P^{(1)},P^{(2)},\ldots
\]
obtained by successive corner cuts with coefficients
\[
    c_0,c_1,c_2,\ldots .
\]

For brevity, set
\[
    A_m=A(P^{(m)}),
    \qquad
    B_m=B(P^{(m)}).
\]
Then the previous section gives
\[
    A_{m+1}
    =
    A_m+\frac{c_m^2}{2}B_m,
\]
and
\[
    B_{m+1}
    =
    2c_m(1-2c_m)B_m.
\]

From the second formula we immediately obtain
\[
    B_m
    =
    \left(
        \prod_{j=0}^{m-1}2c_j(1-2c_j)
    \right)B_0.
\]
In particular, for \(m=0\) this formula gives \(B_0=B(P)\).

Substituting this expression into the recurrence for the area, we get
\[
    A_{m+1}-A_m
    =
    \frac{c_m^2}{2}
    \left(
        \prod_{j=0}^{m-1}2c_j(1-2c_j)
    \right)B(P).
\]
Thus after \(M\) steps,
\[
    A_M
    =
    A(P)
    +
    \left(
        \sum_{m=0}^{M-1}
        \frac{c_m^2}{2}
        \prod_{j=0}^{m-1}2c_j(1-2c_j)
    \right)
    B(P).
\]

Now introduce the constant
\[
    C(c_\bullet)
    =
    \sum_{m=0}^{\infty}
    \frac{c_m^2}{2}
    \prod_{j=0}^{m-1}2c_j(1-2c_j).
\]
It depends only on the sequence of cutting coefficients
\[
    c_\bullet=(c_0,c_1,c_2,\ldots)
\]
and not on the initial polygon.

The series defining \(C(c_\bullet)\) converges. Indeed, since
\(0<c_m<\frac12\),
\[
    0<2c_m(1-2c_m)\leq \frac14.
\]
Moreover,
\[
    0<c_m^2<\frac14.
\]
Therefore
\[
    0
    \leq
    \frac{c_m^2}{2}
    \prod_{j=0}^{m-1}2c_j(1-2c_j)
    \leq
    \frac18\left(\frac14\right)^m,
\]
and the series converges by comparison with a geometric progression.

\begin{theorem}
Let
\[
    P=(p_0,\ldots,p_{n-1})
\]
be the initial oriented polygon in the plane, and let
\[
    c_0,c_1,c_2,\ldots,\qquad 0<c_m<\frac12,
\]
be a sequence of cutting coefficients. Then the oriented area of the limit curve
is
\[
\boxed{
    A_\infty
    =
    A(P)
    +
    C(c_\bullet)
    \sum_{i=0}^{n-1}
    [p_{i-1}-p_i,\;p_{i+1}-p_i]
}.
\]
\end{theorem}

\begin{proof}
This follows from the statements proved above.
\end{proof}

Thus all dependence on the initial polygon is contained in two quantities: its
oriented area \(A(P)\) and the local sum
\[
    \sum_{i=0}^{n-1}
    [p_{i-1}-p_i,\;p_{i+1}-p_i].
\]
All dependence on the chosen corner-cutting scheme is contained in one number
\(C(c_\bullet)\).

\begin{remark}
The theorem is useful not only as a way to compute the area of the limit curve.
It can also be read in reverse. Suppose that, for some sequence of coefficients
\(c_\bullet\), we can independently prove that the limit curve coincides with a
known parametrized curve \(\gamma\). Then its area can be found in the usual
way, for instance by the integral
\[
    \frac12\int [\gamma(t),\gamma'(t)]\,dt
\]
or, in the case of a graph, as the area under the graph. On the other hand, the
same area is expressed by the theorem through the constant \(C(c_\bullet)\), and
this constant is the series
\[
    C(c_\bullet)
    =
    \sum_{m=0}^{\infty}
    \frac{c_m^2}{2}
    \prod_{j=0}^{m-1}2c_j(1-2c_j).
\]
This produces a relation between the usual integral giving the area of the known
curve and a series built from the cutting coefficients. Thus in the examples
below we shall be interested precisely in curves for which both tasks can be
carried out: identifying the limit curve explicitly and choosing the sequence of
coefficients \(c_m\) explicitly.
\end{remark}

The same constant has a local geometric interpretation. Consider the standard
angle with vertex at the origin and neighboring points
\[
    (0,1),\qquad (0,0),\qquad (1,0).
\]
For it the local determinant is
\[
    [(0,1),(1,0)]=-1.
\]
By the one-step formula, the oriented area changes by \(-c_0^2/2\), so the
ordinary positive area of the first cut-off triangle is \(c_0^2/2\). After the
first step the same local determinant is multiplied by \(2c_0(1-2c_0)\), after
the second step by another factor \(2c_1(1-2c_1)\), and so on. Therefore the
full positive area that the limit curve cuts off from the standard angle is
\[
    \sum_{m=0}^{\infty}
    \frac{c_m^2}{2}
    \prod_{j=0}^{m-1}2c_j(1-2c_j)
    =
    C(c_\bullet).
\]
\begin{wrapfigure}[10]{r}{0.31\textwidth}
  \vspace{-1.2\baselineskip}
  \centering
  \begin{tikzpicture}[
      x=3.4cm,
      y=3.4cm,
      line cap=round,
      point/.style={circle,fill=black,inner sep=1pt},
      cutpoint/.style={circle,fill=white,draw=black,inner sep=1pt},
      angle/.style={line width=0.85pt,black!65},
      curve/.style={line width=1pt,blue!70!black},
      stripe/.style={line width=0.28pt,blue!35},
      every label/.append style={font=\scriptsize}
    ]
    \fill[blue!14]
      (0,0) -- (0,0.5)
      plot[domain=0:1,samples=80,variable=\x]
        ({0.5*\x*\x},{0.5*(1-\x)*(1-\x)})
      -- (0.5,0) -- cycle;
    \foreach \s in {0.12,0.24,0.36,0.48,0.60,0.72,0.84} {
      \draw[stripe]
        ({0.5*\s*\s},0) -- ({0.5*\s*\s},{0.5*(1-\s)*(1-\s)});
    }

    \draw[angle] (0,1) -- (0,0) -- (1,0);
    \draw[curve,domain=0:1,samples=100,variable=\x]
      plot ({0.5*\x*\x},{0.5*(1-\x)*(1-\x)});

    \node[point,label=left:{\((0,1)\)}] at (0,1) {};
    \node[point,label=below left:{\((0,0)\)}] at (0,0) {};
    \node[point,label=below:{\((1,0)\)}] at (1,0) {};
    \node[cutpoint,label=left:{\(P_0\)}] at (0,0.5) {};
    \node[cutpoint,label=below:{\(P_2\)}] at (0.5,0) {};
    \node[blue!60!black,font=\scriptsize] at (0.19,0.16) {\(C(1/4)\)};
  \end{tikzpicture}
  \vspace{-0.5\baselineskip}
\end{wrapfigure}
For example, for the stationary choice \(c_m=1/4\), the local limit curve is a
parabola; the shaded region on the right is the area under this parabola inside
the standard angle, and its area is precisely \(C(1/4)\). In other words, in the
standard angle the constant \(C(c_\bullet)\) is literally the area of the local
cut. For an arbitrary vertex, this reduces to the standard angle by an affine
transformation sending the coordinate rays to the two adjacent edges. Oriented
area is multiplied by the corresponding determinant under such a transformation,
so in the general formula the same universal constant stands in front of the
local factor \([p_{i-1}-p_i,\;p_{i+1}-p_i]\) with the appropriate orientation
sign. This explains why the same constant can be found either as the sum of the
areas of all generations of cut-off triangles or through the area of a previously
recognized limit curve in the standard angle.

\section{Examples and Experiments}

Finally, let us see what the formula gives in concrete examples. We begin with
the stationary scheme, in which the cutting coefficient does not change from
step to step.

\subsection{de Rham Curves}

In this note, by de Rham curves we mean the stationary case of the scheme
described above: the same cutting coefficient is used at every step. In other
words,
\[
c_m=c,
\qquad
0<c<\frac12.
\]
Then for all \(m\geq 0\),
\[
B_m=\bigl(2c(1-2c)\bigr)^m B(P).
\]
Consequently,
\[
A_{m+1}-A_m
=
\frac{c^2}{2}
\bigl(2c(1-2c)\bigr)^m B(P).
\]

Thus the changes of area form a geometric progression. After \(M\) steps,
\[
A_M
=
A(P)
+
\frac{c^2}{2}
\sum_{m=0}^{M-1}
\bigl(2c(1-2c)\bigr)^m B(P).
\]
Passing to the limit and summing this progression gives
\[
A_\infty
=
A(P)
+
\frac{c^2}{2\bigl(1-2c+4c^2\bigr)}B(P).
\]

Thus in the stationary case the constant from the main theorem is
\[
C(c_\bullet)
=
\frac{c^2}{2\bigl(1-2c+4c^2\bigr)}.
\]
Therefore the area formula becomes
\[
A_\infty
=
A(P)
+
\frac{c^2}{2\bigl(1-2c+4c^2\bigr)}
\sum_{i=0}^{n-1}
[p_{i-1}-p_i,\;p_{i+1}-p_i].
\]

The same geometric progression gives the exact error formula:
\[
A_\infty-A_M
=
\frac{c^2\bigl(2c(1-2c)\bigr)^M}
     {2\bigl(1-2c+4c^2\bigr)}
B(P).
\]
Thus the areas converge geometrically to the limiting area.

The most famous stationary case is obtained for
\[
c=\frac14.
\]
It is called the Chaikin scheme.

General results on convergence of corner-cutting procedures can be found in de
Boor \cite{deBoor1987}. The classical results of de Rham \cite{deRham1956} show
that for small cutting coefficients these curves are indeed first-order smooth:
for \(0<c<1/3\) the limit curve is differentiable, and in the modern formulation
the boundary case \(c=1/3\) also remains \(C^1\), although the derivative has
zero Holder exponent. At the same time, the case \(c=1/4\) is exceptional: it
produces a spline made of parabolic arcs. For other values of \(c\), the curve
is not piecewise analytic; in particular, among these curves an analytic spline
appears only in the Chaikin case.

\subsection{Parabolas and the Chaikin Scheme}

The best-known stationary case is obtained for
\[
c=\frac14.
\]
It was proposed by George Chaikin in 1974 as a fast algorithm for constructing
curves from a given control polygon \cite{Chaikin1974}. The original emphasis
was not on an analytic description of the limit curve, but on the computational
simplicity of the procedure: the algorithm was recursive and used only simple
operations, convenient for the computer hardware of the time.

On each edge with endpoints \(p_i\) and \(p_{i+1}\), two new points are
constructed:
\[
\frac34p_i+\frac14p_{i+1},
\qquad
\frac14p_i+\frac34p_{i+1}.
\]
In other words, at each step the middle half of every edge is retained, while
corners near the old vertices are cut off.

Soon after the algorithm appeared, Riesenfeld showed that its limit curves are
quadratic splines and are closely related to uniform quadratic \(B\)-splines
\cite{Riesenfeld1975,LaneRiesenfeld1980}. In particular, each local piece of the
limit curve is a quadratic Bezier curve, and hence, in the nondegenerate case,
is an arc of a parabola.

Let us spell out this terminology. A quadratic Bezier curve~\cite{Farin2002}
with control points \(P_0,P_1,P_2\) is given by
\[
    \beta(t)=(1-t)^2P_0+2t(1-t)P_1+t^2P_2,
    \qquad 0\leq t\leq 1.
\]
This is not a new geometric object, but a convenient way to write an arc of a
parabola through its control triangle. Indeed,
\[
\beta(t)
=
P_0+2t(P_1-P_0)+t^2(P_0-2P_1+P_2).
\]
If the points \(P_0,P_1,P_2\) are not collinear, then in affine coordinates with
basis vectors \(P_1-P_0\) and \(P_0-2P_1+P_2\), this curve has the form
\[
    (x,y)=(2t,t^2),
\]
that is, it is the parabola \(y=x^2/4\). If the control points are collinear, one
gets a degenerate case: a segment or a point. Thus in our local nondegenerate
case the words ``quadratic Bezier curve'' can simply be read as ``a parabolic
arc given by a control triangle''.

The main geometric property needed for the Chaikin scheme is shown in
Figure~\ref{fig:bezier-parabola-tangent}. Let
\[
    Q=(1-t)P_0+tP_1,
    \qquad
    M=(1-t)P_1+tP_2.
\]
Then
\[
    T=(1-t)Q+tM=\beta(t),
    \qquad
    \beta'(t)=2(M-Q).
\]
Hence the segment \(QM\) is tangent to the parabola at the point \(T\), and the
points \(Q\), \(M\), and \(T\) divide their respective segments in the same
ratio. The same construction splits the original arc into two smaller Bezier
arcs: the left one has control points \(P_0,Q,T\), and the right one has control
points \(T,M,P_2\).

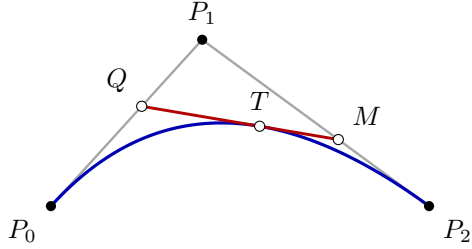
\begin{figure}[ht]
  \centering
  \begin{tikzpicture}[
      x=1cm,
      y=1cm,
      line cap=round,
      point/.style={circle,fill=black,inner sep=1.35pt},
      marked/.style={circle,fill=white,draw=black,inner sep=1.35pt},
      control/.style={line width=0.9pt,black!35},
      tangent/.style={line width=1.1pt,red!70!black},
      curve/.style={line width=1.2pt,blue!70!black},
      every label/.append style={font=\small}
    ]
    \coordinate (P0) at (0,0);
    \coordinate (P1) at (2,2.2);
    \coordinate (P2) at (5,0);
    \coordinate (Q) at (1.2,1.32);
    \coordinate (M) at (3.8,0.88);
    \coordinate (T) at (2.76,1.056);

    \draw[control] (P0) -- (P1) -- (P2);
    \draw[curve,domain=0:1,samples=100,variable=\x]
      plot ({4*\x+\x*\x},{4.4*\x*(1-\x)});
    \draw[tangent] (Q) -- (M);

    \node[point,label=below left:{\(P_0\)}] at (P0) {};
    \node[point,label=above:{\(P_1\)}] at (P1) {};
    \node[point,label=below right:{\(P_2\)}] at (P2) {};
    \node[marked,label=above left:{\(Q\)}] at (Q) {};
    \node[marked,label=above:{\(T\)}] at (T) {};
    \node[marked,label=above right:{\(M\)}] at (M) {};
  \end{tikzpicture}
  \caption{A standard property of a parabola in the form of a quadratic Bezier
  curve: the points \(Q\) and \(M\) divide the sides of the control triangle in
  the same ratio, and the segment \(QM\) is tangent to the parabola at the point
  \(T\), which divides \(QM\) in the same ratio.}
  \label{fig:bezier-parabola-tangent}
\end{figure}

Now apply this to the Chaikin scheme. On the local piece between the midpoints
of two adjacent old edges, take the control triangle \(P_0,P_1,P_2\), where
\(P_1\) is the old vertex and \(P_0\), \(P_2\) are the midpoints of the two
adjacent edges. For \(c=1/4\), the points replacing the vertex \(P_1\) are
exactly the midpoints of the segments \(P_0P_1\) and \(P_1P_2\). By the property
above, the segment joining them is tangent to the corresponding parabola, and
its midpoint \(T\) lies on the parabola. After the next step the same picture is
repeated on the left and right control triangles \(P_0,Q,T\) and \(T,M,P_2\),
then on their halves, and so on, because the triangles \(P_0QT\) and \(TMP_2\)
satisfy the same property. Therefore the Chaikin limit curve coincides with this
quadratic Bezier curve at all dyadic rational parameter values. These points are
everywhere dense by Lemma~\ref{lem:dense-comparison}. Since both curves are
continuous, the local limit and the corresponding parabolic arc coincide
everywhere.

Let us now see what the area formula gives in this case. For \(c=1/4\),
\[
2c(1-2c)=\frac14,
\]
so
\[
C\left(\frac14\right)
=
\frac1{32}
\sum_{m=0}^{\infty}\frac1{4^m}.
\]
Of course, by recalling the usual formula for the sum of a geometric progression
we would immediately obtain the exact value. But let us imagine that this sum is
not known in advance, and compute the same constant directly as the area under
the limiting parabola.

Consider the standard angle with vertex at the origin and neighboring points
\[
(0,1),\qquad (0,0),\qquad (1,0).
\]
The corresponding local piece of the limit curve connects the midpoints of the
two sides,
\[
P_0=\left(0,\frac12\right),
\qquad
P_2=\left(\frac12,0\right),
\]
and is a quadratic Bezier curve with intermediate control point
\[
P_1=(0,0).
\]
Its parametrization is
\[
\gamma(t)
=
(1-t)^2P_0+2t(1-t)P_1+t^2P_2,
\qquad 0\leq t\leq1.
\]
Hence
\[
x(t)=\frac{t^2}{2},
\qquad
 y(t)=\frac{(1-t)^2}{2}.
\]

By the local meaning of the constant, \(C(1/4)\) is the area of the region
between this arc and the coordinate axes. Therefore
\[
C\left(\frac14\right)
=
\int_0^1 y(t)x'(t)\,dt
=
\frac12\int_0^1 t(1-t)^2\,dt
=
\frac1{24}.
\]
On the other hand,
\[
C\left(\frac14\right)
=
\frac1{32}
\sum_{m=0}^{\infty}\frac1{4^m}.
\]
Comparing the two expressions for the same area, we get
\[
\frac1{32}
\sum_{m=0}^{\infty}\frac1{4^m}
=
\frac1{24},
\]
and therefore
\[
\sum_{m=0}^{\infty}\frac1{4^m}
=
\frac43.
\]

Thus, if for some reason the formula for the sum of a geometric progression were
unknown to us, its special case with ratio \(1/4\) could be recovered from the
area of the Chaikin limit curve.

A classical geometric prototype appears in Archimedes' \emph{Quadrature of the
Parabola}: the area of a parabolic segment is \(\frac43\) times the area of a
certain inscribed triangle \cite{ArchimedesQuadrature}. In his construction,
successive generations of triangles have total areas
\[
T,\qquad
\frac{T}{4},\qquad
\frac{T}{16},\qquad
\frac{T}{64},\ldots
\]
and thus lead to the same progression
\[
1+\frac14+\frac1{16}+\frac1{64}+\cdots=\frac43.
\]

In our normalized example, the area of the first cut-off triangle is \(1/32\),
and the full area between the initial angle and the limiting parabola is
\(1/24\). Their ratio is again
\[
\frac{1/24}{1/32}=\frac43.
\]
Thus the appearance of the same constant here is not accidental: both in
Archimedes' quadrature and in the Chaikin scheme, the area of a parabolic region
is assembled from successive generations of triangles decreasing in the same
proportion.

\subsection{Regular Polygons and Approximation of the Circle}

So far the cutting coefficients have been prescribed in advance. We now choose
them in the opposite direction: starting from a regular triangle, we want each
step to send the regular \(3\cdot2^m\)-gon circumscribed about a fixed circle to
the regular \(3\cdot2^{m+1}\)-gon circumscribed about the same circle. Thus we
expect the sequence
\[
3,\ 6,\ 12,\ 24,\ldots
\]
of circumscribed regular polygons, converging to the circle.

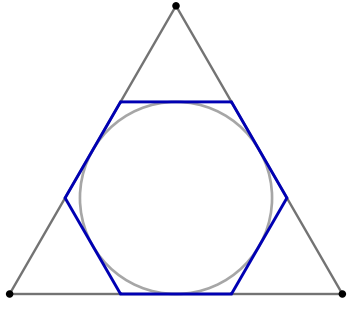
\begin{figure}[ht]
  \centering
  \begin{tikzpicture}[
      x=2.2cm,
      y=2.2cm,
      line cap=round,
      triangle/.style={line width=0.9pt,black!55},
      refined/.style={line width=1.1pt,blue!70!black},
      incircle/.style={line width=1.0pt,gray!70},
      point/.style={circle,fill=black,inner sep=1.0pt}
    ]
    \coordinate (A) at (0,0);
    \coordinate (B) at (2,0);
    \coordinate (C) at (1,1.73205);
    \coordinate (O) at (1,0.57735);
    \coordinate (ACa) at (0.33333,0.57735);
    \coordinate (ACb) at (0.66667,1.15470);
    \coordinate (CBa) at (1.33333,1.15470);
    \coordinate (CBb) at (1.66667,0.57735);
    \coordinate (BAa) at (1.33333,0);
    \coordinate (BAb) at (0.66667,0);

    \draw[triangle] (A) -- (C) -- (B) -- cycle;
    \draw[incircle] (O) circle[radius=0.57735];
    \draw[refined]
      (ACa) -- (ACb) -- (CBa) -- (CBb) -- (BAa) -- (BAb) -- cycle;
    \foreach \pointname in {A,B,C} \node[point] at (\pointname) {};
  \end{tikzpicture}
  \caption{An equilateral triangle, its incircle, and the first cutting step,
  which gives a regular hexagon circumscribed about the same circle.}
  \label{fig:equilateral-incircle}
\end{figure}

Let us find the coefficient forced by this requirement. Suppose that after
\(m\) steps we have a regular
\[
N_m=3\cdot2^m
\]
gon with side length \(\ell_m\). After one more cutting step the new
\(2N_m\)-gon has two types of edges.

On each old edge there remains a middle segment of length
\[
(1-2c_m)\ell_m.
\]
The new edge near an old vertex is the base of an isosceles triangle. Since the
interior angle of a regular \(N_m\)-gon is
\[
\pi-\frac{2\pi}{N_m},
\]
the length of this base is
\[
2c_m\ell_m\cos\frac{\pi}{N_m}.
\]

For the next polygon to be regular, these two lengths must be equal:
\[
1-2c_m
=
2c_m\cos\frac{\pi}{N_m}.
\]
Therefore
\[
c_m
=
\frac{1}{2\left(1+\cos\frac{\pi}{N_m}\right)}
=
\frac{1}{4\cos^2\left(\frac{\pi}{2N_m}\right)}.
\]
Since \(N_m=3\cdot2^m\), this gives
\[
c_m
=
\frac{1}{
4\cos^2\left(\dfrac{\pi}{3\cdot2^{m+1}}\right)
}.
\]

This coefficient realizes the desired passage from the regular \(N_m\)-gon to
the regular \(2N_m\)-gon circumscribed about the same circle.

The same coefficients can be written without trigonometric functions. Set
\[
\Delta_0=3,
\qquad
\Delta_m=
\underbrace{2+\sqrt{2+\cdots+\sqrt{2+\sqrt3}}}_{m\ \text{twos and roots}}
\quad(m\ge1).
\]
This nested expression is obtained by elementary trigonometry: consider
half-angle identity
\[
4\cos^2\frac{\theta}{2}=2+2\cos\theta
\]
applied repeatedly, starting from \(2\cos(\pi/6)=\sqrt3\). The same
half-angle formula gives
\[
\Delta_m=4\cos^2\frac{\pi}{3\cdot2^{m+1}},
\qquad
c_m=\frac1{\Delta_m}.
\]
Equivalently,
\[
c_0=\frac13,
\qquad
c_{m+1}
=
\frac{\sqrt{c_m}}{1+2\sqrt{c_m}}.
\]

It remains only to apply the area formula. For the initial equilateral triangle
with side length \(1\),
\[
A(P)=\frac{\sqrt3}{4},
\qquad
B(P)=-\frac{3\sqrt3}{2}.
\]
Its incircle has radius \(\sqrt3/6\), so the limit area is \(\pi/12\). Hence
\[
\frac{\pi}{12}
=
\frac{\sqrt3}{4}
-
\frac{3\sqrt3}{2}C(c_\bullet),
\qquad
\pi=3\sqrt3-18\sqrt3\,C(c_\bullet).
\]
Substituting the definition of \(C(c_\bullet)\) and \(c_m=1/\Delta_m\), we get
the trigonometry-free series
\[
\boxed{
\pi
=
3\sqrt3
-
9\sqrt3
\sum_{m=0}^{\infty}
\frac{1}{\Delta_m^2}
\prod_{j=0}^{m-1}
\frac{2(\Delta_j-2)}{\Delta_j^2}
},
\]
where the empty product is equal to one.
If the radicals are substituted explicitly, the same identity takes the form
\[
\resizebox{\textwidth}{!}{%
\(\displaystyle
\pi
=
3\sqrt3
-
9\sqrt3
\sum_{m=0}^{\infty}
\frac{1}{
\left(
\underbrace{2+\sqrt{2+\cdots+\sqrt{2+\sqrt3}}}_{m\ \text{twos and roots}}
\right)^2
}
\prod_{j=0}^{m-1}
\frac{
2\left(
\underbrace{2+\sqrt{2+\cdots+\sqrt{2+\sqrt3}}}_{j\ \text{twos and roots}}
-2
\right)
}{
\left(
\underbrace{2+\sqrt{2+\cdots+\sqrt{2+\sqrt3}}}_{j\ \text{twos and roots}}
\right)^2
}.
\)%
}
\]
Here an expression with zero twos and roots is understood to be \(3\). The first
few terms expand as
\[
\resizebox{\textwidth}{!}{%
\(\displaystyle
\pi
=
3\sqrt3
-
9\sqrt3\biggl(
\frac19
+
\frac{2}{9(2+\sqrt3)^2}
+
\frac{2}{9\left(2+\sqrt{2+\sqrt3}\right)^2}
\frac{2\sqrt3}{(2+\sqrt3)^2}
+
\frac{2}{9\left(2+\sqrt{2+\sqrt{2+\sqrt3}}\right)^2}
\frac{2\sqrt3}{(2+\sqrt3)^2}
\frac{2\sqrt{2+\sqrt3}}{\left(2+\sqrt{2+\sqrt3}\right)^2}
+
\cdots
\biggr).
\)%
}
\]
Thus \(\pi\) is expressed using only arithmetic operations and successive square
roots.

If the series is cut off after \(M\) terms, the error is asymptotic to
\[
\frac{\pi^3}{27}\,4^{-M},
\]
so each new step approximately divides the error by four.

The same construction works for an arbitrary regular \(n\)-gon. In this case
\[
N_m=n2^m
\]
and one should choose
\[
c_m^{(n)}
=
\frac{1}{
4\cos^2\left(\dfrac{\pi}{n2^{m+1}}\right)
}.
\]
Then after \(m\) steps one obtains a regular \(n2^m\)-gon circumscribed about
the same circle.

For a regular \(n\)-gon with side length \(\ell\),
\[
B(P)
=
-n\ell^2\sin\frac{2\pi}{n}.
\]
Since
\[
A(P)
=
\frac{n\ell^2}{4\tan(\pi/n)},
\]
we get
\[
B(P)
=
-8\sin^2\frac{\pi}{n}\,A(P).
\]
Consequently,
\[
A_\infty
=
\left(
1-8C(c_\bullet)\sin^2\frac{\pi}{n}
\right)A(P).
\]

If the initial regular \(n\)-gon is circumscribed about a circle of radius \(r\),
then
\[
A(P)=nr^2\tan\frac{\pi}{n},
\qquad
A_\infty=\pi r^2.
\]
Therefore, for the corresponding sequence of coefficients,
\[
C(c_\bullet^{(n)})
=
\frac{
1-\dfrac{\pi}{n\tan(\pi/n)}
}{
8\sin^2(\pi/n)
}.
\]

Finally, let us note that the same example also gives an elliptic variant. The
corner-cutting scheme is affinely invariant: if a nondegenerate affine
transformation is applied to the initial polyline, then the same division
parameters are preserved on each edge, and all subsequent polylines and the
limit curve pass to their affine images. Therefore an affine image of the circle
constructed above gives an ellipse tangent to the sides of an arbitrary
nondegenerate triangle at their midpoints. In the standard terminology of
triangle geometry, this particular ellipse is the \emph{Steiner inellipse}: the
unique ellipse inscribed in the triangle and tangent to the three sides at their
midpoints \cite{MindaPhelps2008}.

\subsection{The hyperbola and an algebraic series for \texorpdfstring{\(\log 2\)}{log 2}}

We now pass from circular symmetry to a more general principle. In the case of
the circle, all local configurations of a regular circumscribed polygon are the
same up to rotation. For the hyperbola, the role of rotations will be played by
another group of linear transformations. What matters is not that the curve is a
circle, but that it is mapped to itself, its tangents are mapped to tangents,
and local configurations are carried to one another while ratios on lines are
preserved~\cite{Salmon1879}. Then the cutting coefficient is enough to find in
one standard configuration, after which the same coefficient may be used on the
entire refinement level.

Consider the branch of the hyperbola
\[
\mathcal H
=
\{(x,y)\in\mathbb R^2:xy=1,\ x>0,\ y>0\}
\]
between the points
\[
(1,1)
\qquad\text{and}\qquad
\left(2,\frac12\right).
\]
The area under this arc is
\[
\int_1^2\frac{dx}{x}=\log 2.
\]
Our goal is to choose the corner-cutting coefficients so that the
corresponding limit curve is precisely this arc of the hyperbola. Then the area
formula will give an expression for the logarithm of two.

A convenient parametrization of the hyperbola is
\[
H(t)
=
\begin{pmatrix}
e^t\\
e^{-t}
\end{pmatrix}.
\]
Indeed,
\[
H(0)=(1,1),
\qquad
H(\log 2)=\left(2,\frac12\right).
\]

The construction relies on the one-parameter group of linear
transformations
\[
G_s
=
\begin{pmatrix}
e^s&0\\
0&e^{-s}
\end{pmatrix}.
\]
Explicitly,
\[
G_s
\begin{pmatrix}
x\\
y
\end{pmatrix}
=
\begin{pmatrix}
e^sx\\
e^{-s}y
\end{pmatrix}.
\]

Let us explain precisely which invariance property is being used.
If \((x,y)\in\mathcal H\), then
\[
(e^sx)(e^{-s}y)=xy=1.
\]
Hence
\[
G_s(\mathcal H)=\mathcal H.
\]
Moreover,
\[
H(t+s)=G_sH(t).
\]
Thus, translation of the parameter \(t\mapsto t+s\) is realized by a
linear transformation of the plane.

The tangent line to the hyperbola at the point \(H(t)\) has equation
\[
\ell_t:\qquad e^{-t}x+e^ty=2.
\]
The transformations \(G_s\) preserve the entire family of tangent
lines. Indeed, if \((x,y)\in\ell_t\), then
\[
e^{-(t+s)}e^sx+e^{t+s}e^{-s}y
=
e^{-t}x+e^ty
=
2,
\]
and therefore
\[
G_s(\ell_t)=\ell_{t+s}.
\]

This is the invariance that makes a uniform corner-cutting rule
possible. Every local configuration
\[
\ell_{t-h},\qquad \ell_t,\qquad \ell_{t+h}
\]
is transformed by \(G_{-t}\) into the standard configuration
\[
\ell_{-h},\qquad \ell_0,\qquad \ell_h.
\]
Since affine transformations preserve ratios in which points divide
a line segment, the corner-cutting coefficient depends only on the
parameter step \(h\), and not on the position \(t\) along the
hyperbola. Consequently, one common coefficient can be used at every
vertex of a given refinement level.

Fix \(h>0\) and consider the tangent lines
\[
\ell_{kh},
\qquad k\in\mathbb Z.
\]
Let
\[
q_k=\ell_{kh}\cap\ell_{(k+1)h}
\]
be the intersection points of consecutive tangents. The edge
\[
[q_{k-1},q_k]
\]
lies on the tangent line \(\ell_{kh}\).

We first show that the point of tangency \(H(kh)\) is the midpoint of
this edge. By applying \(G_{-kh}\), it is enough to verify the claim
for \(k=0\).

The tangent line \(\ell_0\) has equation
\[
x+y=2.
\]
Parametrize it by
\[
(x,y)=(1+u,1-u).
\]
Substitution into the equation of \(\ell_h\) gives
\[
e^{-h}(1+u)+e^h(1-u)=2,
\]
and hence
\[
u
=
\frac{\cosh h-1}{\sinh h}
=
\tanh\frac h2.
\]
Therefore, the point \(\ell_0\cap\ell_h\) corresponds to
\[
u=\tanh\frac h2,
\]
whereas \(\ell_{-h}\cap\ell_0\) corresponds to
\[
u=-\tanh\frac h2.
\]
Their midpoint corresponds to \(u=0\), and is therefore equal to
\[
(1,1)=H(0).
\]
Thus, for every \(k\),
\[
H(kh)=\frac{q_{k-1}+q_k}{2}.
\]

We now refine the tangent polygon by halving the parameter step.
Between the old tangent lines we insert the new tangents with
parameters
\[
\left(k+\frac12\right)h.
\]
On the edge lying on \(\ell_0\), the endpoints of the old edge have
coordinates
\[
u=-a,\qquad u=a,
\qquad
a=\tanh\frac h2,
\]
while the two new points have coordinates
\[
u=-b,\qquad u=b,
\qquad
b=\tanh\frac h4.
\]

Let \(c\) be the corner-cutting coefficient. The left new point must
satisfy
\[
-b=(1-c)(-a)+ca.
\]
It follows that
\[
c=\frac{a-b}{2a}.
\]
Using the double-angle identity
\[
a=\tanh\frac h2
=
\frac{2b}{1+b^2},
\]
we obtain
\[
c
=
\frac{1-b^2}{4}
=
\frac{1}{4\cosh^2(h/4)}.
\]

Therefore, the tangent polygon with parameter step \(h\) is transformed
into the tangent polygon with parameter step \(h/2\) by the
corner-cutting coefficient
\[
c(h)=\frac{1}{4\cosh^2(h/4)}.
\]

For the arc under consideration, the initial parameter step is
\[
h_0=\log 2.
\]
After \(m\) refinement steps,
\[
h_m=\frac{\log 2}{2^m}.
\]
Hence the required coefficients are
\[
c_m
=
\frac{1}{
4\cosh^2\left(\dfrac{\log 2}{2^{m+2}}\right)
}.
\]
With this choice, the polygon obtained after \(m\) steps is exactly the
tangent polygon with parameter step \((\log 2)/2^m\). Since this step
tends to zero, the polygons converge to the hyperbolic arc.

The hyperbolic functions can now be eliminated. Set
\[
a_m=\sqrt[2^{m+1}]{2}.
\]
Then
\[
c_m
=
\frac{a_m}{(1+a_m)^2}
=
\frac{\sqrt[2^{m+1}]{2}}
{\left(1+\sqrt[2^{m+1}]{2}\right)^2}.
\]
In particular,
\[
c_0
=
\frac{\sqrt2}{(1+\sqrt2)^2}
=
3\sqrt2-4.
\]

The same sequence can also be defined recursively. From the
half-argument identity for the hyperbolic cosine we obtain
\[
c_{m+1}
=
\frac{\sqrt{c_m}}{1+2\sqrt{c_m}}.
\]
Thus,
\[
c_0=3\sqrt2-4,
\qquad
c_{m+1}
=
\frac{\sqrt{c_m}}{1+2\sqrt{c_m}},
\]
which defines all coefficients using only arithmetic operations and
square roots.

It remains to apply the area formula. Let
\[
L=\log 2.
\]
Three consecutive vertices of the initial tangent polygon are
\[
q_{-1}
=
\ell_{-L}\cap\ell_0
=
\left(\frac23,\frac43\right),
\]
\[
q_0
=
\ell_0\cap\ell_L
=
\left(\frac43,\frac23\right),
\]
and
\[
q_1
=
\ell_L\cap\ell_{2L}
=
\left(\frac83,\frac13\right).
\]
The midpoints of the two adjacent edges are
\[
\frac{q_{-1}+q_0}{2}=(1,1),
\qquad
\frac{q_0+q_1}{2}
=
\left(2,\frac12\right).
\]

The initial local polygonal arc between these two points consists of
two tangent segments. The first lies on
\[
y=2-x,
\]
and the second on
\[
y=1-\frac{x}{4}.
\]
The area below this polygonal arc is
\[
\int_1^{4/3}(2-x)\,dx
+
\int_{4/3}^{2}
\left(1-\frac{x}{4}\right)\,dx
=
\frac23.
\]
The area below the limiting hyperbolic arc is
\[
\int_1^2\frac{dx}{x}
=
\log 2.
\]

Although the arc considered here is not closed, the local form of the
area formula applies. Equivalently, one may close both the initial and
the limiting arcs by the same fixed polygonal path; its contribution
then cancels when the two areas are compared.

Furthermore,
\[
[q_{-1}-q_0,\;q_1-q_0]
=
-\frac23.
\]
Taking the orientation into account, the area formula gives
\[
\log 2-\frac23
=
\frac23\,C(c_\bullet).
\]
Hence
\[
\log 2
=
\frac23+\frac23\,C(c_\bullet),
\]
and therefore
\[
\log 2
=
\frac23
+
\frac13
\sum_{m=0}^{\infty}
c_m^2
\prod_{j=0}^{m-1}
2c_j(1-2c_j).
\]

Finally, substitute
\[
c_m=\frac{a_m}{(1+a_m)^2},
\qquad
a_m=\sqrt[2^{m+1}]{2}.
\]
Since
\[
a_{m+1}^2=a_m,
\]
the summand simplifies to
\[
c_m^2
\prod_{j=0}^{m-1}
2c_j(1-2c_j)
=
\frac{
3\cdot 2^{m+1}(a_m-1)^3
}{
(a_m+1)(a_m^2+1)
}.
\]

We therefore arrive at the purely algebraic identity
\[
\boxed{
\log 2
=
\frac23
+
\sum_{m=0}^{\infty}
\frac{
2^{m+1}
\left(\sqrt[2^{m+1}]{2}-1\right)^3
}{
\left(\sqrt[2^{m+1}]{2}+1\right)
\left(\sqrt[2^m]{2}+1\right)
}
}.
\]

\section{Conclusion and Further Questions}

We have considered symmetric nonstationary corner-cutting schemes in which, on
each refinement level, the same coefficient is applied to all vertices, while
the coefficient itself may vary from step to step. For such schemes we proved
the existence of a continuous limit curve and the convergence of the oriented
areas of the approximating polygons to the area of this limit.

The main result is the formula
\[
A_\infty=A(P)+C(c_\bullet)B(P),
\]
where the quantities \(A(P)\) and \(B(P)\) depend only on the initial oriented
polygon, while the constant
\[
C(c_\bullet)
=
\sum_{m=0}^{\infty}
\frac{c_m^2}{2}
\prod_{j=0}^{m-1}2c_j(1-2c_j)
\]
depends only on the sequence of cutting coefficients. Thus the geometric data
of the initial contour and the rule of its successive refinement enter the
answer independently of one another. The local interpretation of this constant
as the area of a cut in the standard angle explains the affine nature of the
formula.

The examples considered above show two sides of the result. In the stationary
case the formula is summed directly and, in particular, describes the area
change for de Rham curves and the Chaikin scheme. Conversely, a known limit
curve allows one to compute \(C(c_\bullet)\) geometrically. This gives
representations of the numbers \(\pi\) and \(\log 2\) as series determined by
successive cuts. Thus corner cutting may be viewed not only as a method for
constructing curves, but also as a geometric mechanism for decomposing area into
successive local contributions.

A natural continuation would be to study the regularity of the limit curve for
general nonstationary sequences \(c_\bullet\). Another interesting direction is
the inverse problem: to describe sequences of coefficients for which the limit
coincides with a prescribed curve. In particular, one is led to the problem of
classifying curves whose family of tangents is compatible with a single cutting
coefficient on each refinement level. Other possible directions include schemes
with vertex-dependent coefficients and higher-dimensional analogues of the area
formula for truncating polyhedra.

\section*{Acknowledgments}
\addcontentsline{toc}{section}{Acknowledgments}

The author thanks N.~N.~Kosovskii for suggesting the topic of this work,
D.~V.~Karpov for help with terminology at an early stage, and
A.~D.~Baranov for consultations on subtleties of convergence.

\phantomsection

\end{document}


\maketitle

\begin{abstract}
Рассматриваются схемы отсечения углов для замкнутых ориентированных
многоугольников на плоскости. На каждом шаге каждая
вершина заменяется двумя точками на соседних рёбрах; один и тот же коэффициент
отсечения используется во всех вершинах текущего многоугольника, но этот
коэффициент может меняться от шага к шагу. Для таких схем доказывается
существование непрерывной предельной кривой и получается явная формула для её
ориентированной площади. Основной результат состоит в разделении двух вкладов:
геометрия исходного контура входит через его ориентированную площадь и локальную
сумму определителей, построенных по соседним рёбрам, а вся зависимость от
правила уточнения содержится в одной числовой константе. Дополнительно
обсуждается геометрический смысл этой константы как площади, возникающей в
стандартном угле. В качестве приложений рассматриваются стационарные кривые
де Рама, схема Чайкина, связь с квадратурой параболы Архимеда, а также
специальные последовательности коэффициентов, приводящие к окружности и к дуге
гиперболы; из них получаются соответствующие ряды для \(\pi\) и \(\log 2\).
\end{abstract}

\tableofcontents

\section{Введение}

Рассмотрим замкнутый многоугольный контур: грубый набросок будущей
кривой. Чтобы сделать его более гладким, можно срезать каждый угол. На двух
рёбрах, сходящихся в вершине, выбираются две новые точки, а сама вершина
заменяется коротким отрезком между ними. После одного шага получается новый
многоугольник с большим числом сторон; после многих шагов контур выглядит всё
более гладким.

\begin{figure}[ht]
  \centering
  \begin{tikzpicture}[
      x=1cm,
      y=1cm,
      line cap=round,
      vertex/.style={circle,fill=black,inner sep=1.4pt},
      cutpoint/.style={circle,fill=white,draw=black,inner sep=1.3pt},
      original/.style={line width=1.1pt,black},
      first/.style={line width=1.1pt,blue!70!black},
      second/.style={line width=1.1pt,red!70!black},
      continuation/.style={line width=0.9pt,densely dashed,opacity=0.55},
      label/.style={font=\small,anchor=east,align=right},
      row/.style={xshift=3.3cm}
    ]
    \node[label] at (2.50,0.70) {исходная\\ломаная};
    \begin{scope}[row]
      \draw[original,continuation] (-0.55,-0.46) -- (0,0);
      \draw[original] (0,0) -- (1.6,1.35) -- (3.4,0.15) -- (5.0,1.25);
      \draw[original,continuation] (5.0,1.25) -- (5.55,1.63);
      \foreach \x/\y in {0/0,1.6/1.35,3.4/0.15,5.0/1.25} \node[vertex] at (\x,\y) {};
    \end{scope}

    \node[label] at (2.50,-1.45) {после одного\\шага};
    \begin{scope}[row,yshift=-2.15cm]
      \draw[black!25,line width=0.8pt] (0,0) -- (1.6,1.35) -- (3.4,0.15) -- (5.0,1.25);
      \draw[first,continuation] (0.25,0.21) -- (0.70,0.59);
      \draw[first] (0.70,0.59) -- (1.15,0.97) -- (2.10,1.01) -- (2.90,0.49) -- (3.85,0.46) -- (4.30,0.77);
      \draw[first,continuation] (4.30,0.77) -- (4.75,1.08);
      \foreach \x/\y in {1.15/0.97,2.10/1.01,2.90/0.49,3.85/0.46} \node[cutpoint] at (\x,\y) {};
    \end{scope}

    \node[label] at (2.50,-3.60) {после двух\\шагов};
    \begin{scope}[row,yshift=-4.30cm]
      \draw[black!20,line width=0.8pt] (0,0) -- (1.6,1.35) -- (3.4,0.15) -- (5.0,1.25);
      \draw[blue!25,line width=0.8pt] (1.15,0.97) -- (2.10,1.01) -- (2.90,0.49) -- (3.85,0.46);
      \draw[second,continuation] (0.38,0.32) -- (0.60,0.51);
      \draw[second] (0.60,0.51) -- (0.83,0.70) -- (1.42,0.98) -- (1.84,1.00) -- (2.33,0.87) -- (2.67,0.63) -- (3.16,0.48) -- (3.58,0.47) -- (4.17,0.68) -- (4.39,0.83);
      \draw[second,continuation] (4.39,0.83) -- (4.62,0.99);
      \foreach \x/\y in {0.83/0.70,1.42/0.98,1.84/1.00,2.33/0.87,2.67/0.63,3.16/0.48,3.58/0.47,4.17/0.68} \node[cutpoint] at (\x,\y) {};
    \end{scope}
  \end{tikzpicture}
  \caption{Локальная картина отсечения углов на фрагменте из трёх звеньев.}
  \label{fig:local-corner-cutting}
\end{figure}

Возникает естественный вопрос: можно ли найти площадь предельной кривой, не
умея явно записать саму эту кривую? Оказывается, для рассматриваемых схем это
возможно. Более того, вся геометрия исходного многоугольника описывается всего
двумя величинами, а бесконечная последовательность отсечений --- одной
универсальной константой. В конце статьи мы увидим, как из этого возникают
геометрические ряды для \(\pi\) и \(\log 2\).

Именно так устроены схемы отсечения углов. Они являются классическим способом
строить предельные кривые по многоугольным контрольным данным; в современной
терминологии это один из простейших примеров схем подразделения
\cite{DynLevin2002}. Самый известный пример такой процедуры --- алгоритм
Чайкина, пределом которого является квадратичная \(B\)-сплайновая\begingroup\renewcommand{\thefootnote}{\fnsymbol{footnote}}%
\footnote[1]{Сплайн --- это кривая, склеенная из простых
полиномиальных кусков. Квадратичный сплайн означает, что каждый такой кусок
задаётся многочленами степени не выше двух; геометрически это дуга параболы
или её вырожденный случай. \(B\)-сплайн --- не новый вид кривой, а удобный
способ записывать такие гладко склеенные куски через контрольные точки;
буква \(B\) означает \emph{basis}, то есть базис. Поэтому в этой заметке
фразу «квадратичный \(B\)-сплайн» можно читать как «гладко склеенная
цепочка кусков парабол».}\endgroup{} кривая, то есть, грубо говоря,
гладко склеенная цепочка кусков парабол.

В этой заметке рассматривается особенно простой, но достаточно гибкий вариант:
на каждом уровне уточнения один и тот же коэффициент используется для всех
углов. При переходе к следующему уровню этот коэффициент разрешается изменить.
Такой выбор сохраняет одинаковую обработку всех вершин внутри одного шага и
одновременно позволяет менять правило уточнения от уровня к уровню.

Нас интересует не столько гладкость предельной кривой, сколько её площадь. Даже
если саму кривую трудно записать явно, её площадь можно найти напрямую из
исходного многоугольника и коэффициентов отсечения. Здесь используется
ориентированная площадь: она учитывает направление обхода контура, поэтому при
смене ориентации меняет знак.

Основной результат можно прочитать так: площадь предельной кривой равна площади
исходного многоугольника плюс поправка, которая устроена очень просто. Если
\(P=(p_0,\ldots,p_{n-1})\) --- исходный ориентированный многоугольник, а
\((c_k)_{k\ge 0}\) --- последовательность коэффициентов отсечения, то
ориентированная площадь предельной кривой имеет вид
\[
  A_\infty
  =
  A(P)
  +
  C(c_\bullet)
  \sum_i
  \det(p_{i-1}-p_i,\ p_{i+1}-p_i),
\]
Здесь \(A(P)\) --- ориентированная площадь исходного многоугольника. Сумма
определителей описывает локальные повороты исходного контура около его вершин, а
число \(C(c_\bullet)\) зависит только от выбранной последовательности отсечений.
Константа, в свою очередь, имеет вид
\[
  C(c_\bullet)
  =
  \sum_{k=0}^{\infty}
  2^k\frac{c_k^2}{2}
  \prod_{j=0}^{k-1}c_j(1-2c_j),
\]
где пустое произведение при \(k=0\) равно \(1\). Иными словами, геометрия
начального многоугольника и правила сглаживания входят в формулу раздельно.

Для чтения статьи достаточно знания определителя двух векторов, формулы
площади многоугольника и элементарных свойств рядов. В отдельных примерах
используются параметризация кривой и простейшие интегралы.

\section{Схема отсечения углов}

Перейдём от картинки к точным обозначениям. В дальнейшем многоугольник будет
пониматься как упорядоченный циклический набор точек на плоскости. То есть
запись
\[
    P=(p_0,p_1,\ldots,p_{n-1})
\]
означает, что вершины соединяются в указанном порядке, а после \(p_{n-1}\)
снова идёт \(p_0\). Многоугольник может быть невыпуклым; для формулы площади
важна именно ориентация обхода.

Все индексы ниже считаются циклически. Если у многоугольника \(N\) вершин, то
\(p_N=p_0\), \(p_{-1}=p_{N-1}\), и так далее. Это соглашение позволяет не
выносить отдельно крайние случаи.

Зафиксируем последовательность коэффициентов отсечения
\[
    c_0,c_1,c_2,\ldots,
    \qquad
    0<c_m<\frac12.
\]
Коэффициент \(c_m\) используется на \(m\)-м шаге для всех вершин одновременно.
Иными словами, в пределах одного шага правило одинаково для всех вершин, но
между шагами оно может изменяться.

Опишем один шаг. Пусть уже построен ориентированный многоугольник
\[
    P^{(m)}=\bigl(p^{(m)}_0,p^{(m)}_1,\ldots,p^{(m)}_{N_m-1}\bigr).
\]
На каждом его ребре
\([p^{(m)}_i,p^{(m)}_{i+1}]
\) берём две точки
\[
    p^{(m+1)}_{2i}
    =
    (1-c_m)p^{(m)}_i+c_m p^{(m)}_{i+1},
\]
\[
    p^{(m+1)}_{2i+1}
    =
    c_m p^{(m)}_i+(1-c_m)p^{(m)}_{i+1}.
\]
Первая из них лежит ближе к началу старого ребра, вторая --- ближе к его концу.
После этого старые вершины удаляются, а все новые точки соединяются в порядке
номеров. Полученный многоугольник обозначим через \(P^{(m+1)}\).

Эквивалентно, у каждой старой вершины \(p^{(m)}_i\) появляется короткое новое
ребро, соединяющее две точки на соседних старых рёбрах. Именно это ребро и есть
«срезанный угол». Каждое старое ребро даёт две новые вершины, поэтому
\[
    N_m=n2^m.
\]

Условие \(0<c_m<1/2\) означает, что отсечение не вырождается. При \(c_m=0\)
форма многоугольника фактически не изменилась бы, а при \(c_m=1/2\) две новые
точки на каждом старом ребре совпали бы в его середине. Эти предельные случаи
полезны для интуиции, но в основной схеме мы их не включаем.

Для произвольного \(c\in(0,1/2)\) обозначим через \(T_c\) оператор одного
шага, задаваемый описанным выше правилом. Тогда вся последовательность
многоугольников задаётся рекурсивно:
\[
    P^{(m+1)}=T_{c_m}P^{(m)},
    \qquad
    P^{(0)}=P.
\]
Или, что то же самое,
\[
    P^{(m)}
    =
    T_{c_{m-1}}\circ T_{c_{m-2}}\circ\cdots\circ T_{c_0}(P),
    \qquad m\ge 1.
\]

Ниже мы проверим, что при бесконечном повторении этой процедуры возникает
непрерывная предельная кривая. Её будем обозначать \(P^{(\infty)}\). Нас будет интересовать как её параметрическое уравнение, так и
ориентированная площадь, которую она ограничивает.

\section{Ориентированная площадь многоугольников}

В этом разделе зафиксируем обозначения для площади. Для двух векторов
\(u=(u_1,u_2)\) и \(v=(v_1,v_2)\) будем писать
\[
    [u,v]
    =
    \det
    \begin{pmatrix}
        u_1 & u_2 \\
        v_1 & v_2
    \end{pmatrix}
    =
    u_1v_2-u_2v_1.
\]
Геометрически \([u,v]\) равен ориентированной площади параллелограмма,
построенного на векторах \(u\) и \(v\).

Дальше в этом разделе буква \(Q\) будет обозначать произвольный
ориентированный многоугольник. Это удобно отличать от исходного многоугольника
\(P\): позже полученные формулы будут применяться к текущим многоугольникам
\(Q=P^{(m)}\).

Итак, пусть
\[
    Q=(q_0,\ldots,q_{N-1})
\]
--- ориентированный многоугольник на плоскости. Его ориентированная площадь
определяется формулой
\[
    A(Q)
    =
    \frac12
    \sum_{i=0}^{N-1}
    [q_i,q_{i+1}],
\]
где, как обычно, \(q_N=q_0\).

Эта формула известна как формула Гаусса, или формула шнурования\begingroup\renewcommand{\thefootnote}{\fnsymbol{footnote}}%
\footnote[1]{Формула является частным случаем теоремы Грина, а та, в
свою очередь, является частным случаем теоремы Стокса. Тем, кто с ними
не знаком, всегда доступен метод индукции.}\endgroup. Если
координаты вершин равны \(q_i=(x_i,y_i)\), то она принимает вид
\[
    A(Q)
    =
    \frac12
    \sum_{i=0}^{N-1}
    (x_i y_{i+1}-y_i x_{i+1}).
\]

Важно, что \(A(Q)\) является именно ориентированной площадью. Если вершины
многоугольника обходятся против часовой стрелки, то для обычного выпуклого
многоугольника она положительна. Если поменять направление обхода, знак
площади изменится:
\[
    A(q_{N-1},q_{N-2},\ldots,q_0)=-A(Q).
\]
Для самопересекающихся многоугольников эта величина также имеет смысл, но
уже как алгебраическая площадь: области, обходимые с разной ориентацией,
учитываются с разными знаками.

Нам также понадобится локальная величина, связанная с одной вершиной
многоугольника. Для вершины \(q_i\) рассмотрим два вектора, выходящие из неё
к соседним вершинам:
\[
    q_{i-1}-q_i,
    \qquad
    q_{i+1}-q_i.
\]
Их определитель
\[
    [q_{i-1}-q_i,\;q_{i+1}-q_i]
\]
равен удвоенной ориентированной площади треугольника с вершинами
\(q_{i-1},q_i,q_{i+1}\). Для удобства также введём величину
\[
    B(Q)
    =
    \sum_{i=0}^{N-1}
    [q_{i-1}-q_i,\;q_{i+1}-q_i].
\]

Она собирает по всем вершинам многоугольника определители, построенные на
двух соседних с данной вершиной рёбрах. Геометрически каждое слагаемое равно
удвоенной ориентированной площади треугольника с вершинами
\(q_{i-1},q_i,q_{i+1}\). Впрочем, эта интерпретация не будет для нас существенной:
можно считать, что \(B(Q)\) вводится просто как сокращение, делающее дальнейшие
выкладки более читаемыми.

Именно величина \(B(Q)\) будет появляться в формуле изменения площади при
одном шаге отсечения углов. В частности, для исходного многоугольника
\(P=(p_0,\ldots,p_{n-1})\) будем использовать обозначение
\[
    B(P)
    =
    \sum_{i=0}^{n-1}
    [p_{i-1}-p_i,\;p_{i+1}-p_i].
\]

\section{Предельная кривая и сходимость площадей}

Разберемся с тем, почему наши выкладки корректны. Именно, почему предельная кривая действительно существует
и почему её площадь можно получать как предел площадей многоугольников.

Обозначим через
\[
    s_i=\frac{p_i+p_{i+1}}2
\]
середину ребра \([p_i,p_{i+1}]\) исходного многоугольника. Эти точки удобно
использовать как границы локальных кусков: часть кривой между \(s_{i-1}\) и
\(s_i\) строится только по трём соседним вершинам
\(p_{i-1},p_i,p_{i+1}\). Это следует из локальности схемы отсечения углов.

\begin{lemma}
Последовательность многоугольников \(P^{(m)}\) имеет непрерывную предельную
замкнутую кривую \(P^{(\infty)}\) в метрике Хаусдорфа. Более того, если
\[
    A_m=A(P^{(m)}),
\]
то ориентированная площадь предельной кривой равна пределу ориентированных
площадей аппроксимирующих многоугольников:
\[
    A_\infty=\lim_{m\to\infty}A_m.
\]
\end{lemma}

\begin{proof}
Докажем утверждение локально, на куске между \(s_{i-1}\) и \(s_i\). Если точки
\(p_{i-1},p_i,p_{i+1}\) лежат на одной прямой, то все последующие локальные
куски также лежат на этой прямой, и утверждение очевидно. Поэтому будем считать
соответствующий треугольник невырожденным.

Сделаем аффинную замену координат, переводящую отрезок
\([s_{i-1},s_i]\) в отрезок \([0,1]\) оси \(x\), а вершину \(p_i\) --- в точку
верхней полуплоскости. После такой замены каждый локальный кусок
\(P^{(m)}\) является графиком некоторой кусочно-линейной непрерывной функции
\(f_{i,m}\colon[0,1]\to\mathbb R\). При переходе от уровня \(m\) к уровню
\(m+1\) мы только срезаем углы предыдущего графика. Поэтому
\[
    0\leq f_{i,m+1}(x)\leq f_{i,m}(x)\leq f_{i,0}(x)
    \qquad x\in[0,1].
\]
Следовательно, для каждого \(x\) существует предел
\[
    f_i(x)=\lim_{m\to\infty}f_{i,m}(x).
\]

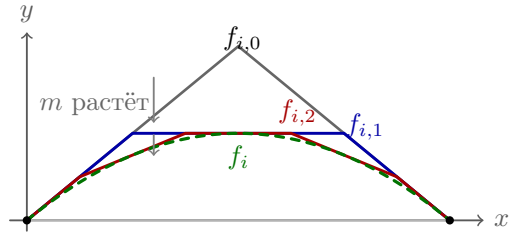
\begin{figure}[ht]
  \centering
  \begin{tikzpicture}[
      x=5.6cm,
      y=2.3cm,
      line cap=round,
      axis/.style={->,line width=0.7pt,black!65},
      base/.style={line width=0.7pt,black!30},
      fzero/.style={line width=1.0pt,black!60},
      fone/.style={line width=1.1pt,blue!65!black},
      ftwo/.style={line width=1.1pt,red!65!black},
      flimit/.style={line width=1.2pt,green!45!black,dashed},
      point/.style={circle,fill=black,inner sep=1.1pt},
      label/.style={font=\small}
    ]
    \draw[axis] (-0.04,0) -- (1.08,0) node[right,label] {\(x\)};
    \draw[axis] (0,-0.06) -- (0,1.08) node[above,label] {\(y\)};
    \draw[base] (0,0) -- (1,0);

    \draw[fzero]
      (0,0) -- (0.5,1.0) -- (1,0);
    \draw[fone]
      (0,0) -- (0.25,0.5) -- (0.75,0.5) -- (1,0);
    \draw[ftwo]
      (0,0) -- (0.125,0.25) -- (0.375,0.5) -- (0.625,0.5)
      -- (0.875,0.25) -- (1,0);
    \draw[flimit]
      plot[domain=0:1,samples=80] (\x,{2*\x*(1-\x)});

    \node[point,label={[label]below left:\(s_{i-1}\)}] at (0,0) {};
    \node[point,label={[label]below:\(s_i\)}] at (1,0) {};
    \node[label] at (0.51,1.04) {\(f_{i,0}\)};
    \node[label,blue!65!black] at (0.80,0.54) {\(f_{i,1}\)};
    \node[label,red!65!black] at (0.64,0.62) {\(f_{i,2}\)};
    \node[label,green!45!black] at (0.50,0.36) {\(f_i\)};

    \draw[->,black!50,line width=0.7pt] (0.30,0.82) -- (0.30,0.56);
    \draw[->,black!50,line width=0.7pt] (0.30,0.49) -- (0.30,0.37);
    \node[label,align=center,black!60] at (0.16,0.67)
      {\(m\) растёт};
  \end{tikzpicture}
  \caption{Локальная сходимость графиков для примера \(c_m\equiv 1/4\):
  \(f_{i,0}\) --- исходная ломаная, \(f_{i,1}\) и \(f_{i,2}\) --- первые два
  шага отсечения углов, а пунктиром показан предельный график.}
  \label{fig:local-graph-convergence}
\end{figure}

Кроме того, функции \(f_{i,m}\) выпуклы вверх в том же смысле, что и исходная
ломаная на этом треугольнике: каждая следующая получается из предыдущей
конечным числом отсечений прямыми. Значит и предельная функция \(f_i\) выпукла
вверх. Такая функция непрерывна на внутренности отрезка
\cite{BoltyanskyYaglom1951}, а в концах непрерывность следует из оценки
\[
    0\leq f_i(x)\leq f_{i,0}(x),
\]
поскольку \(f_{i,0}(0)=f_{i,0}(1)=0\). Теперь функции \(f_{i,m}\)
непрерывны на компактном отрезке \([0,1]\), монотонно убывают к непрерывной
функции \(f_i\), поэтому по теореме Дини сходимость является равномерной.
Равномерная сходимость влечёт сходимость графиков в метрике Хаусдорфа. Итак,
каждый локальный предельный кусок является непрерывной кривой. Соседние куски
имеют общие концы \(s_i\), поэтому из конечного числа таких кусков склеивается
непрерывная замкнутая кривая \(P^{(\infty)}\).

Осталось проверить сходимость площадей. В выбранных аффинных координатах
ориентированная площадь локального куска, замкнутого отрезком
\([s_{i-1},s_i]\), отличается от интеграла
\[
    \int_0^1 f_{i,m}(x)\,dx
\]
только постоянным ориентированным множителем, равным определителю обратного
аффинного преобразования. При этом
\[
    |f_{i,m}(x)|\leq f_{i,0}(x),
\]
а функция \(f_{i,0}\) интегрируема. Поэтому по теореме Лебега о мажорированной
сходимости
\[
    \int_0^1 f_{i,m}(x)\,dx
    \longrightarrow
    \int_0^1 f_i(x)\,dx.
\]

Площадь всего многоугольника раскладывается в сумму площади фиксированного
многоугольника с вершинами \(s_0,\ldots,s_{n-1}\) и конечного числа таких
локальных вкладов. Для предельной кривой получается такое же разложение, только
с функциями \(f_i\). Так как локальных кусков конечное число, можно перейти к
пределу в каждом из них и просуммировать результаты. Следовательно,
\[
    A(P^{(m)})\longrightarrow A(P^{(\infty)}).
\]
Эту предельную ориентированную площадь мы и обозначаем через \(A_\infty\).
\end{proof}

Следующая простая лемма будет использоваться в примерах, когда мы будем
сравнивать уже известную гладкую кривую с предельной кривой, полученной
отсечением углов.

\begin{lemma}[Сравнение по плотному множеству]\label{lem:dense-comparison}
Пусть \(\gamma,\widetilde\gamma\colon[a,b]\to\mathbb R^2\) --- две
непрерывные кривые. Если они совпадают на всюду плотном множестве
\(D\subset[a,b]\), то они совпадают на всём отрезке.

Один удобный способ получить такое множество \(D\) таков. Пусть
\(\Pi_m\) --- вложенная последовательность разбиений отрезка \([a,b]\),
а \( |\Pi_m| \) --- длина самого большого промежутка разбиения. Если
\[
    |\Pi_m|\longrightarrow 0,
\]
то множество всех точек всех разбиений
\[
    D=\bigcup_{m\geq0}\Pi_m
\]
всюду плотно в \([a,b]\). В частности, достаточно, чтобы при переходе от
\(\Pi_m\) к \(\Pi_{m+1}\) каждый старый промежуток заменялся промежутками
длины не больше \(\lambda_m\) от его длины, где \(0<\lambda_m<1\) и
\[
    \prod_{j=0}^{m-1}\lambda_j\longrightarrow0.
\]
Например, это заведомо выполнено, если каждый промежуток делится одной
точкой в отношении \(\theta_m:(1-\theta_m)\) и все
\(\theta_m\) отделены от \(0\) и \(1\):
\[
    \eta\leq\theta_m\leq1-\eta
    \qquad\text{для некоторого }\eta>0.
\]
\end{lemma}

\begin{proof}
Стандартный факт о совпадении непрерывных отображений на всюду плотном
множестве будем считать известным. Остаётся проверить плотность множества
точек разбиений. Возьмём любой открытый
интервал \((u,v)\subset[a,b]\). Если \(|\Pi_m|<v-u\), то внутри
\((u,v)\) обязательно есть точка разбиения \(\Pi_m\); иначе весь интервал
\((u,v)\) лежал бы внутри одного промежутка разбиения длины не меньше
\(v-u\), что противоречит выбору \(m\). Значит каждый открытый интервал
пересекает \(D\), то есть \(D\) всюду плотно.

Наконец, если каждый новый промежуток имеет длину не больше \(\lambda_m\)
от длины старого, то
\[
    |\Pi_m|\leq
    (b-a)\prod_{j=0}^{m-1}\lambda_j.
\]
Поэтому при указанном условии \(|\Pi_m|\to0\).
\end{proof}

\section{Изменение площади за один шаг}

Теперь вычислим, как меняется ориентированная площадь при одном шаге
отсечения углов.

Пусть
\[
    Q=(q_0,\ldots,q_{N-1})
\]
--- некоторый ориентированный многоугольник, а \(T_cQ\) --- многоугольник,
полученный из него отсечением углов с коэффициентом \(c\), где
\[
    0<c<\frac12.
\]
Обозначим вершины нового многоугольника через
\[
    r_0,\ldots,r_{2N-1}.
\]
По определению схемы имеем
\[
    r_{2i}=(1-c)q_i+cq_{i+1},
\]
\[
    r_{2i+1}=cq_i+(1-c)q_{i+1}.
\]

\begin{lemma}
Для любого ориентированного многоугольника \(Q\) выполнено равенство
\[
    A(T_cQ)
    =
    A(Q)+\frac{c^2}{2}B(Q).
\]
\end{lemma}

\begin{proof}
По формуле ориентированной площади,
\[
    2A(T_cQ)
    =
    \sum_{i=0}^{N-1}
    \bigl([r_{2i},r_{2i+1}]+[r_{2i+1},r_{2i+2}]\bigr).
\]
Подставим выражения для новых вершин.

Сначала
\[
    [r_{2i},r_{2i+1}]
    =
    [(1-c)q_i+cq_{i+1},\;cq_i+(1-c)q_{i+1}].
\]
По билинейности и кососимметричности определителя получаем
\[
    [r_{2i},r_{2i+1}]
    =
    (1-2c)[q_i,q_{i+1}].
\]

Далее
\[
    [r_{2i+1},r_{2i+2}]
    =
    [cq_i+(1-c)q_{i+1},\;(1-c)q_{i+1}+cq_{i+2}].
\]
Раскрывая определитель, получаем
\[
    [r_{2i+1},r_{2i+2}]
    =
    c(1-c)[q_i,q_{i+1}]
    +c^2[q_i,q_{i+2}]
    +c(1-c)[q_{i+1},q_{i+2}].
\]

Теперь суммируем по всем \(i\). Последняя сумма после сдвига индекса совпадает
с суммой \(\sum_i [q_i,q_{i+1}]\). Поэтому
\[
    2A(T_cQ)
    =
    (1-2c^2)\sum_{i=0}^{N-1}[q_i,q_{i+1}]
    +
    c^2\sum_{i=0}^{N-1}[q_i,q_{i+2}].
\]
Так как
\[
    2A(Q)=\sum_{i=0}^{N-1}[q_i,q_{i+1}],
\]
получаем
\[
    2A(T_cQ)-2A(Q)
    =
    c^2
    \left(
        \sum_{i=0}^{N-1}[q_i,q_{i+2}]
        -
        2\sum_{i=0}^{N-1}[q_i,q_{i+1}]
    \right).
\]

Осталось заметить, что выражение в скобках равно \(B(Q)\). Действительно,
\[
    B(Q)
    =
    \sum_{i=0}^{N-1}
    [q_{i-1}-q_i,\;q_{i+1}-q_i].
\]
Раскрывая определитель, имеем
\[
    [q_{i-1}-q_i,\;q_{i+1}-q_i]
    =
    [q_{i-1},q_{i+1}]
    -
    [q_{i-1},q_i]
    -
    [q_i,q_{i+1}].
\]
После суммирования и сдвига индексов получаем
\[
    B(Q)
    =
    \sum_{i=0}^{N-1}[q_i,q_{i+2}]
    -
    2\sum_{i=0}^{N-1}[q_i,q_{i+1}].
\]
Следовательно,
\[
    2A(T_cQ)-2A(Q)=c^2B(Q),
\]
то есть
\[
    A(T_cQ)
    =
    A(Q)+\frac{c^2}{2}B(Q).
\]
\end{proof}

Полученная формула имеет простой геометрический смысл. При отсечении углов
площадь меняется на сумму площадей маленьких треугольников, отрезанных около
старых вершин. Знак этой поправки зависит от ориентации обхода. Например, для
выпуклого многоугольника, обходимого против часовой стрелки, величина \(B(Q)\)
отрицательна, поэтому площадь уменьшается.

Для дальнейшего нам понадобится ещё одно равенство: как при отсечении углов
меняется сама величина \(B(Q)\).

\begin{lemma}
Для любого ориентированного многоугольника \(Q\) выполнено равенство
\[
    B(T_cQ)=2c(1-2c)B(Q).
\]
\end{lemma}

\begin{proof}
Пусть, как и выше,
\[
    r_{2i}=(1-c)q_i+cq_{i+1},
    \qquad
    r_{2i+1}=cq_i+(1-c)q_{i+1}.
\]
В сумму \(B(T_cQ)\) входят слагаемые двух типов: около вершин \(r_{2i}\) и
около вершин \(r_{2i+1}\).

Для вершины \(r_{2i}\) имеем
\[
    r_{2i-1}-r_{2i}
    =
    c(q_{i-1}-q_{i+1}),
\]
\[
    r_{2i+1}-r_{2i}
    =
    (1-2c)(q_{i+1}-q_i).
\]
Значит соответствующее слагаемое равно
\[
    c(1-2c)
    [q_{i-1}-q_{i+1},\;q_{i+1}-q_i].
\]

Для вершины \(r_{2i+1}\) аналогично:
\[
    r_{2i}-r_{2i+1}
    =
    (1-2c)(q_i-q_{i+1}),
\]
\[
    r_{2i+2}-r_{2i+1}
    =
    c(q_{i+2}-q_i).
\]
Поэтому второе слагаемое равно
\[
    c(1-2c)
    [q_i-q_{i+1},\;q_{i+2}-q_i].
\]

Суммируя по всем \(i\), получаем
\[
    B(T_cQ)
    =
    c(1-2c)
    \sum_i
    [q_{i-1}-q_{i+1},\;q_{i+1}-q_i]
\]
\[
    +
    c(1-2c)
    \sum_i
    [q_i-q_{i+1},\;q_{i+2}-q_i].
\]
Первая сумма равна \(B(Q)\), поскольку
\[
    q_{i-1}-q_{i+1}
    =
    (q_{i-1}-q_i)+(q_i-q_{i+1}),
\]
а добавочное слагаемое даёт нулевой определитель.

Вторая сумма тоже равна \(B(Q)\) после сдвига индекса. Следовательно,
\[
    B(T_cQ)=c(1-2c)B(Q)+c(1-2c)B(Q),
\]
то есть
\[
    B(T_cQ)=2c(1-2c)B(Q).
\]
\end{proof}

Итак, один шаг отсечения углов задаёт две простые рекуррентные формулы:
\[
    A(T_cQ)=A(Q)+\frac{c^2}{2}B(Q),
\]
\[
    B(T_cQ)=2c(1-2c)B(Q).
\]
Всё готово для получения основного результата.

\section{Основная формула площади}

Теперь применим формулы предыдущего раздела ко всей последовательности
многоугольников
\[
    P^{(0)},P^{(1)},P^{(2)},\ldots
\]
полученной последовательным отсечением углов с коэффициентами
\[
    c_0,c_1,c_2,\ldots .
\]

Для краткости положим
\[
    A_m=A(P^{(m)}),
    \qquad
    B_m=B(P^{(m)}).
\]
Тогда из предыдущего раздела следует, что
\[
    A_{m+1}
    =
    A_m+\frac{c_m^2}{2}B_m,
\]
и
\[
    B_{m+1}
    =
    2c_m(1-2c_m)B_m.
\]

Из второй формулы сразу получаем
\[
    B_m
    =
    \left(
        \prod_{j=0}^{m-1}2c_j(1-2c_j)
    \right)B_0.
\]
В частности, при \(m=0\) эта формула даёт \(B_0=B(P)\).

Подставляя это выражение в рекурсию для площади, получаем
\[
    A_{m+1}-A_m
    =
    \frac{c_m^2}{2}
    \left(
        \prod_{j=0}^{m-1}2c_j(1-2c_j)
    \right)B(P).
\]
Значит после \(M\) шагов
\[
    A_M
    =
    A(P)
    +
    \left(
        \sum_{m=0}^{M-1}
        \frac{c_m^2}{2}
        \prod_{j=0}^{m-1}2c_j(1-2c_j)
    \right)
    B(P).
\]

Теперь введём константу
\[
    C(c_\bullet)
    =
    \sum_{m=0}^{\infty}
    \frac{c_m^2}{2}
    \prod_{j=0}^{m-1}2c_j(1-2c_j).
\]
Она зависит только от последовательности коэффициентов отсечения
\[
    c_\bullet=(c_0,c_1,c_2,\ldots)
\]
и не зависит от исходного многоугольника.

Ряд, задающий \(C(c_\bullet)\), сходится. Действительно, так как
\(0<c_m<\frac12\), то
\[
    0<2c_m(1-2c_m)\leq \frac14.
\]
Кроме того,
\[
    0<c_m^2<\frac14.
\]
Следовательно,
\[
    0
    \leq
    \frac{c_m^2}{2}
    \prod_{j=0}^{m-1}2c_j(1-2c_j)
    \leq
    \frac18\left(\frac14\right)^m,
\]
а значит ряд сходится по сравнению с геометрической прогрессией.

\begin{theorem}
Пусть
\[
    P=(p_0,\ldots,p_{n-1})
\]
--- исходный ориентированный многоугольник на плоскости, а
\[
    c_0,c_1,c_2,\ldots,\qquad 0<c_m<\frac12,
\]
--- последовательность коэффициентов отсечения. Тогда ориентированная площадь
предельной кривой равна
\[
\boxed{
    A_\infty
    =
    A(P)
    +
    C(c_\bullet)
    \sum_{i=0}^{n-1}
    [p_{i-1}-p_i,\;p_{i+1}-p_i]
}.
\]
\end{theorem}

\begin{proof}
Состоит из утверждений, приведённых выше.
\end{proof}

Итак, вся зависимость от исходного многоугольника содержится в двух величинах:
его ориентированной площади \(A(P)\) и локальной сумме
\[
    \sum_{i=0}^{n-1}
    [p_{i-1}-p_i,\;p_{i+1}-p_i].
\]
Вся зависимость от выбранной схемы отсечения содержится в одном числе
\(C(c_\bullet)\).

\begin{remark}
Формула теоремы полезна не только как способ вычислять площадь предельной
кривой. Её можно читать и в обратную сторону. Предположим, что для некоторой
последовательности коэффициентов \(c_\bullet\) мы умеем независимо доказать,
что предельная кривая совпадает с заранее известной параметризованной кривой
\(\gamma\). Тогда её площадь можно найти обычным способом, например через
интеграл
\[
    \frac12\int [\gamma(t),\gamma'(t)]\,dt
\]
или, в случае графика, как площадь под графиком. С другой стороны, та же
площадь выражается по теореме через константу \(C(c_\bullet)\), а эта
константа равна ряду
\[
    C(c_\bullet)
    =
    \sum_{m=0}^{\infty}
    \frac{c_m^2}{2}
    \prod_{j=0}^{m-1}2c_j(1-2c_j).
\]
Так появляется соотношение между обычным интегралом, задающим площадь
известной кривой, и рядом, построенным по коэффициентам отсечения. Поэтому
в примерах ниже нас будут интересовать именно такие кривые, для которых можно
сделать обе вещи: явно узнать предельную кривую и явно подобрать
последовательность коэффициентов \(c_m\).
\end{remark}

\refstepcounter{remark}\label{rem:local-constant-interpretation}
\noindent\textbf{Замечание \theremark\ (локальная интерпретация \(C(c_\bullet)\)).}
У той же константы есть локальное геометрическое чтение. Рассмотрим
стандартный угол с вершиной в начале координат и соседними точками
\[
    (0,1),\qquad (0,0),\qquad (1,0).
\]
Для него локальный определитель равен
\[
    [(0,1),(1,0)]=-1.
\]
По формуле одного шага ориентированная площадь меняется на
\(-c_0^2/2\), то есть обычная положительная площадь первого отсечённого
треугольника равна \(c_0^2/2\). После первого шага тот же локальный
определитель умножается на \(2c_0(1-2c_0)\), после второго --- ещё на
\(2c_1(1-2c_1)\), и так далее. Поэтому полная положительная площадь,
которую предельная кривая отсекает от стандартного угла, равна
\[
    \sum_{m=0}^{\infty}
    \frac{c_m^2}{2}
    \prod_{j=0}^{m-1}2c_j(1-2c_j)
    =
    C(c_\bullet).
\]
\begin{wrapfigure}[10]{r}{0.31\textwidth}
  \vspace{-1.2\baselineskip}
  \centering
  \begin{tikzpicture}[
      x=3.4cm,
      y=3.4cm,
      line cap=round,
      point/.style={circle,fill=black,inner sep=1pt},
      cutpoint/.style={circle,fill=white,draw=black,inner sep=1pt},
      angle/.style={line width=0.85pt,black!65},
      curve/.style={line width=1pt,blue!70!black},
      hatch/.style={line width=0.25pt,blue!40},
      every label/.append style={font=\scriptsize}
    ]
    \foreach \s in {0.10,0.20,0.30,0.40,0.50,0.60,0.70,0.80,0.90} {
      \draw[hatch]
        ({0.5*\s*\s},0) -- ({0.5*\s*\s},{0.5*(1-\s)*(1-\s)});
    }

    \draw[angle] (0,1) -- (0,0) -- (1,0);
    \draw[curve,domain=0:1,samples=100,variable=\x]
      plot ({0.5*\x*\x},{0.5*(1-\x)*(1-\x)});

    \node[point,label=left:{\((0,1)\)}] at (0,1) {};
    \node[point,label=below left:{\((0,0)\)}] at (0,0) {};
    \node[point,label=below:{\((1,0)\)}] at (1,0) {};
    \node[cutpoint,label=left:{\(P_0\)}] at (0,0.5) {};
    \node[cutpoint,label=below:{\(P_2\)}] at (0.5,0) {};
    \node[blue!65!black,font=\scriptsize] at (0.20,0.16) {\(C(c_\bullet)\)};
  \end{tikzpicture}
  \vspace{-0.5\baselineskip}
\end{wrapfigure}
На рисунке штриховкой показан соответствующий локальный срез. Здесь важно
только общее утверждение: в стандартном угле константа
\(C(c_\bullet)\) буквально является площадью локального среза.
Для произвольной вершины это сводится
к стандартному углу аффинным преобразованием, отправляющим координатные
лучи в два соседних ребра. Ориентированная площадь при таком преобразовании умножается на
соответствующий определитель, поэтому в общей формуле та же универсальная
константа стоит перед локальным множителем
\([p_{i-1}-p_i,\;p_{i+1}-p_i]\) с нужным знаком ориентации. Это объясняет,
почему одну и ту же константу можно находить либо как сумму площадей всех
поколений отсечённых треугольников, либо через площадь заранее опознанной
предельной кривой в стандартном угле.

\section{Примеры и эксперименты}

Наконец, посмотрим, что даёт полученная формула на конкретных примерах.
Сначала разберём стационарную схему, в которой коэффициент отсечения
не меняется от шага к шагу.

\subsection{Кривые де Рама}

Под кривыми де Рама в этой заметке будем понимать стационарный
случай описанной выше схемы: на каждом шаге используется один и тот же
коэффициент отсечения. Иными словами,
\[
c_m=c,
\qquad
0<c<\frac12.
\]
Тогда для всех \(m\geq 0\)
\[
B_m=\bigl(2c(1-2c)\bigr)^m B(P).
\]
Следовательно,
\[
A_{m+1}-A_m
=
\frac{c^2}{2}
\bigl(2c(1-2c)\bigr)^m B(P).
\]

Таким образом, изменения площади образуют геометрическую прогрессию.
После \(M\) шагов получаем
\[
A_M
=
A(P)
+
\frac{c^2}{2}
\sum_{m=0}^{M-1}
\bigl(2c(1-2c)\bigr)^m B(P).
\]
Переходя к пределу и суммируя эту прогрессию, получаем
\[
A_\infty
=
A(P)
+
\frac{c^2}{2\bigl(1-2c+4c^2\bigr)}B(P).
\]

Итак, в стационарном случае константа из основной теоремы равна
\[
C(c_\bullet)
=
\frac{c^2}{2\bigl(1-2c+4c^2\bigr)}.
\]
Поэтому формула площади принимает вид
\[
A_\infty
=
A(P)
+
\frac{c^2}{2\bigl(1-2c+4c^2\bigr)}
\sum_{i=0}^{n-1}
[p_{i-1}-p_i,\;p_{i+1}-p_i].
\]

Из той же геометрической прогрессии получаем точную формулу для ошибки:
\[
A_\infty-A_M
=
\frac{c^2\bigl(2c(1-2c)\bigr)^M}
     {2\bigl(1-2c+4c^2\bigr)}
B(P).
\]
Таким образом, площади сходятся к предельной площади геометрически.

Наиболее известный стационарный случай получается при
\[
c=\frac14.
\]
Он и называется схемой Чайкина.

Общие результаты о сходимости процедур отсечения углов можно найти у
де Бура~\cite{deBoor1987}. Классические результаты де Рама~\cite{deRham1956}
показывают, что при малых коэффициентах отсечения эти кривые действительно
становятся гладкими в первом порядке: при \(0<c<1/3\) предельная кривая
дифференцируема, а в современной формулировке граничный случай \(c=1/3\)
также остаётся \(C^1\), хотя производная уже имеет нулевой показатель Гёльдера. При этом случай
\(c=1/4\) является исключительным: тогда получается сплайн,
составленный из дуг парабол. Для остальных \(c\) кривая не является
кусочно-аналитической; в частности, среди этих кривых аналитический сплайн
возникает только в случае Чайкина.

\subsection{Параболы и схема Чайкина}

Наиболее известный стационарный случай получается при
\[
c=\frac14.
\]
Он был предложен Джорджем Чайкиным в 1974 году как быстрый алгоритм
построения кривых по заданному контрольному многоугольнику
\cite{Chaikin1974}. Изначально акцент делался не на аналитическом
описании предельной кривой, а на вычислительной простоте процедуры:
алгоритм был рекурсивным и использовал лишь простые операции,
удобные для компьютерной техники того времени.

На каждом ребре с концами \(p_i\) и \(p_{i+1}\) строятся две новые точки
\[
\frac34p_i+\frac14p_{i+1},
\qquad
\frac14p_i+\frac34p_{i+1}.
\]
Иными словами, на каждом шаге от ребра сохраняется его средняя половина,
а около старых вершин отсекаются углы.

Вскоре после появления алгоритма Ризенфельд показал, что его предельные
кривые являются квадратичными сплайнами и тесно связаны с равномерными
квадратичными \(B\)-сплайнами~\cite{Riesenfeld1975,LaneRiesenfeld1980}.
В частности, каждый локальный кусок предельной кривой является квадратичной
кривой Безье, а значит, в невырожденном случае представляет собой дугу параболы.

Остановимся на этой терминологии подробнее. Квадратичная кривая
Безье~\cite{Farin2002} с контрольными точками \(P_0,P_1,P_2\)
задаётся формулой
\[
    \beta(t)=(1-t)^2P_0+2t(1-t)P_1+t^2P_2,
    \qquad 0\leq t\leq 1.
\]
Это не новый геометрический объект, а удобный способ записывать дугу
параболы через контрольный треугольник. Действительно,
\[
\beta(t)
=
P_0+2t(P_1-P_0)+t^2(P_0-2P_1+P_2).
\]
Если точки \(P_0,P_1,P_2\) не лежат на одной прямой, то в аффинных
координатах с базисными векторами \(P_1-P_0\) и
\(P_0-2P_1+P_2\) эта кривая имеет вид
\[
    (x,y)=(2t,t^2),
\]
то есть является параболой \(y=x^2/4\). Если же контрольные точки
коллинеарны, получается вырожденный случай: отрезок или точка. Поэтому
в нашем локальном невырожденном случае слова «квадратичная кривая Безье»
можно читать просто как «дуга параболы, заданная контрольным треугольником».

Главное геометрическое свойство, которое нужно для схемы Чайкина, показано
на рис.~\ref{fig:bezier-parabola-tangent}. Пусть
\[
    Q=(1-t)P_0+tP_1,
    \qquad
    M=(1-t)P_1+tP_2.
\]
Тогда
\[
    T=(1-t)Q+tM=\beta(t),
    \qquad
    \beta'(t)=2(M-Q).
\]
Значит, отрезок \(QM\) касается параболы в точке \(T\), причём точки
\(Q\), \(M\) и \(T\) делят свои отрезки в одном и том же отношении.
Та же конструкция разбивает исходную дугу на две меньшие дуги Безье:
левая задаётся контрольными точками \(P_0,Q,T\), а правая --- точками
\(T,M,P_2\).

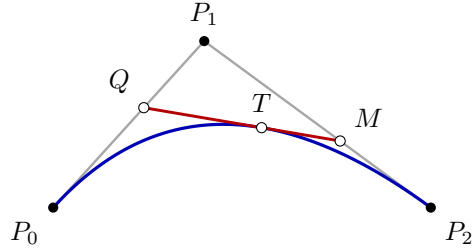
\begin{figure}[ht]
  \centering
  \begin{tikzpicture}[
      x=1cm,
      y=1cm,
      line cap=round,
      point/.style={circle,fill=black,inner sep=1.35pt},
      marked/.style={circle,fill=white,draw=black,inner sep=1.35pt},
      control/.style={line width=0.9pt,black!35},
      tangent/.style={line width=1.1pt,red!70!black},
      curve/.style={line width=1.2pt,blue!70!black},
      every label/.append style={font=\small}
    ]
    \coordinate (P0) at (0,0);
    \coordinate (P1) at (2,2.2);
    \coordinate (P2) at (5,0);
    \coordinate (Q) at (1.2,1.32);
    \coordinate (M) at (3.8,0.88);
    \coordinate (T) at (2.76,1.056);

    \draw[control] (P0) -- (P1) -- (P2);
    \draw[curve,domain=0:1,samples=100,variable=\x]
      plot ({4*\x+\x*\x},{4.4*\x*(1-\x)});
    \draw[tangent] (Q) -- (M);

    \node[point,label=below left:{\(P_0\)}] at (P0) {};
    \node[point,label=above:{\(P_1\)}] at (P1) {};
    \node[point,label=below right:{\(P_2\)}] at (P2) {};
    \node[marked,label=above left:{\(Q\)}] at (Q) {};
    \node[marked,label=above:{\(T\)}] at (T) {};
    \node[marked,label=above right:{\(M\)}] at (M) {};
  \end{tikzpicture}
  \caption{Стандартное свойство параболы в форме квадратичной кривой Безье:
  точки \(Q\) и \(M\) делят стороны контрольного треугольника в одном
  отношении, а отрезок \(QM\) касается параболы в точке \(T\), делящей
  \(QM\) в том же отношении.}
  \label{fig:bezier-parabola-tangent}
\end{figure}

Теперь применим это к схеме Чайкина. На локальном куске между серединами
соседних старых звеньев возьмём контрольный треугольник
\(P_0,P_1,P_2\), где \(P_1\) --- старая вершина, а \(P_0\) и \(P_2\)
--- середины двух соседних звеньев. При \(c=1/4\) точки, которыми
заменяем вершину \(P_1\), являются ровно серединами отрезков
\(P_0P_1\) и \(P_1P_2\). По свойству выше, соединяющий их отрезок
касается соответствующей параболы, а его середина \(T\) лежит на ней.
После следующего шага та же картина повторяется на левом и правом
контрольных треугольниках \(P_0,Q,T\) и \(T,M,P_2\), затем на их
половинах и так далее, поскольку для треугольников \(P_0QT\) и \(TMP_2\) выполняется то же самое свойство. Поэтому предельная кривая Чайкина совпадает с
этой квадратичной кривой Безье на всех двоично-рациональных значениях
параметра. Эти точки всюду плотны по лемме~\ref{lem:dense-comparison}. Поскольку обе кривые непрерывны, локальный предел и соответствующая дуга
параболы совпадают всюду.

Посмотрим теперь, что формула площади даёт в этом случае. Для
\(c=1/4\) имеем
\[
2c(1-2c)=\frac14,
\]
поэтому
\[
C\left(\frac14\right)
=
\frac1{32}
\sum_{m=0}^{\infty}\frac1{4^m}.
\]
Конечно, вспомнив обычную формулу суммы геометрической прогрессии,
мы сразу получили бы точное значение. Но давайте представим,
что сумма прогрессии нам заранее неизвестна, и вычислим ту же константу
непосредственно как площадь под предельной параболой.

Рассмотрим стандартный угол с вершиной в начале координат и соседними
точками
\[
(0,1),\qquad (0,0),\qquad (1,0).
\]
Соответствующий локальный кусок предельной кривой соединяет середины
двух сторон,
\[
P_0=\left(0,\frac12\right),
\qquad
P_2=\left(\frac12,0\right),
\]
и является квадратичной кривой Безье с промежуточной контрольной
точкой
\[
P_1=(0,0).
\]
Её параметризация имеет вид
\[
\gamma(t)
=
(1-t)^2P_0+2t(1-t)P_1+t^2P_2,
\qquad 0\leq t\leq1.
\]
Следовательно,
\[
x(t)=\frac{t^2}{2},
\qquad
y(t)=\frac{(1-t)^2}{2}.
\]

По локальному смыслу константа \(C(1/4)\) равна площади области между
этой дугой и осями координат. Поэтому
\[
C\left(\frac14\right)
=
\int_0^1 y(t)x'(t)\,dt
=
\frac12\int_0^1 t(1-t)^2\,dt
=
\frac1{24}.
\]
С другой стороны,
\[
C\left(\frac14\right)
=
\frac1{32}
\sum_{m=0}^{\infty}\frac1{4^m}.
\]
Сравнивая два выражения для одной и той же площади, получаем
\[
\frac1{32}
\sum_{m=0}^{\infty}\frac1{4^m}
=
\frac1{24},
\]
а значит,
\[
\sum_{m=0}^{\infty}\frac1{4^m}
=
\frac43.
\]

Таким образом, если бы формула суммы геометрической прогрессии по
какой-либо причине была нам неизвестна, её частный случай со знаменателем
\(1/4\) можно было бы восстановить из площади предельной кривой Чайкина.

Это рассуждение имеет древний геометрический прообраз. В трактате
«Квадратура параболы» Архимед доказал, что площадь параболического
сегмента равна \(\frac43\) площади определённого вписанного треугольника
\cite{ArchimedesQuadrature}. В его построении последовательные поколения
треугольников имеют суммарные площади
\[
T,\qquad
\frac{T}{4},\qquad
\frac{T}{16},\qquad
\frac{T}{64},\ldots
\]
и потому приводят к той же прогрессии
\[
1+\frac14+\frac1{16}+\frac1{64}+\cdots=\frac43.
\]

В нашем нормированном примере площадь первого отсечённого треугольника
равна \(1/32\), а полная площадь между начальным углом и предельной
параболой равна \(1/24\). Их отношение снова равно
\[
\frac{1/24}{1/32}=\frac43.
\]
Поэтому появление той же константы здесь не случайно: и в квадратуре
Архимеда, и в схеме Чайкина площадь параболической области собирается
из последовательных поколений треугольников, уменьшающихся в одной
и той же пропорции.

\subsection{Правильные многоугольники и приближение окружности}

До сих пор коэффициенты отсечения задавались заранее. Теперь выберем их в
обратном направлении: начиная с правильного треугольника, мы хотим, чтобы
каждый шаг переводил правильный \(3\cdot2^m\)-угольник, описанный около
фиксированной окружности, в правильный \(3\cdot2^{m+1}\)-угольник,
описанный около той же окружности. Поэтому ожидаем последовательность
\[
3,\ 6,\ 12,\ 24,\ldots
\]
правильных описанных многоугольников, сходящихся к окружности.

\begin{figure}[ht]
  \centering
  \begin{tikzpicture}[
      x=2.2cm,
      y=2.2cm,
      line cap=round,
      triangle/.style={line width=0.9pt,black!55},
      refined/.style={line width=1.1pt,blue!70!black},
      incircle/.style={line width=1.0pt,gray!70},
      point/.style={circle,fill=black,inner sep=1.0pt}
    ]
    \coordinate (A) at (0,0);
    \coordinate (B) at (2,0);
    \coordinate (C) at (1,1.73205);
    \coordinate (O) at (1,0.57735);
    \coordinate (ACa) at (0.33333,0.57735);
    \coordinate (ACb) at (0.66667,1.15470);
    \coordinate (CBa) at (1.33333,1.15470);
    \coordinate (CBb) at (1.66667,0.57735);
    \coordinate (BAa) at (1.33333,0);
    \coordinate (BAb) at (0.66667,0);

    \draw[triangle] (A) -- (C) -- (B) -- cycle;
    \draw[incircle] (O) circle[radius=0.57735];
    \draw[refined]
      (ACa) -- (ACb) -- (CBa) -- (CBb) -- (BAa) -- (BAb) -- cycle;
    \foreach \pointname in {A,B,C} \node[point] at (\pointname) {};
  \end{tikzpicture}
  \caption{Правильный треугольник, его вписанная окружность и первый шаг
  отсечения, дающий правильный шестиугольник, описанный около той же
  окружности.}
  \label{fig:equilateral-incircle}
\end{figure}

Найдём коэффициент, который навязывает это требование. Пусть после \(m\) шагов
у нас есть правильный
\[
N_m=3\cdot2^m
\]
угольник со стороной \(\ell_m\). После ещё одного шага отсечения у нового
\(2N_m\)-угольника появляются рёбра двух типов.

На каждом старом ребре остаётся средний отрезок длины
\[
(1-2c_m)\ell_m.
\]
Новое ребро около старой вершины является основанием равнобедренного
треугольника. Поскольку внутренний угол правильного \(N_m\)-угольника равен
\[
\pi-\frac{2\pi}{N_m},
\]
длина этого основания равна
\[
2c_m\ell_m\cos\frac{\pi}{N_m}.
\]

Чтобы следующий многоугольник был правильным, эти две длины должны совпадать:
\[
1-2c_m
=
2c_m\cos\frac{\pi}{N_m}.
\]
Следовательно,
\[
c_m
=
\frac{1}{2\left(1+\cos\frac{\pi}{N_m}\right)}
=
\frac{1}{4\cos^2\left(\frac{\pi}{2N_m}\right)}.
\]
Так как \(N_m=3\cdot2^m\), получаем
\[
c_m
=
\frac{1}{
4\cos^2\left(\dfrac{\pi}{3\cdot2^{m+1}}\right)
}.
\]

Именно этот коэффициент реализует нужный переход от правильного
\(N_m\)-угольника к правильному \(2N_m\)-угольнику, описанному около той же
окружности.

Те же коэффициенты можно записать без тригонометрических функций. Положим
\[
\Delta_0=3,
\qquad
\Delta_m=
\underbrace{2+\sqrt{2+\cdots+\sqrt{2+\sqrt3}}}_{m\ \text{двоек и корней}}
\quad(m\ge1).
\]
Это вложенное выражение получается из элементарной тригонометрии: достаточно
многократно применить формулу половинного угла
\[
4\cos^2\frac{\theta}{2}=2+2\cos\theta
\]
начиная с \(2\cos(\pi/6)=\sqrt3\). Та же формула половинного угла даёт
\[
\Delta_m=4\cos^2\frac{\pi}{3\cdot2^{m+1}},
\qquad
c_m=\frac1{\Delta_m}.
\]
Эквивалентно,
\[
c_0=\frac13,
\qquad
c_{m+1}
=
\frac{\sqrt{c_m}}{1+2\sqrt{c_m}}.
\]

Остаётся только применить формулу площади. Для начального правильного
треугольника со стороной \(1\)
\[
A(P)=\frac{\sqrt3}{4},
\qquad
B(P)=-\frac{3\sqrt3}{2}.
\]
Радиус его вписанной окружности равен \(\sqrt3/6\), поэтому предельная площадь
равна \(\pi/12\). Следовательно,
\[
\frac{\pi}{12}
=
\frac{\sqrt3}{4}
-
\frac{3\sqrt3}{2}C(c_\bullet),
\qquad
\pi=3\sqrt3-18\sqrt3\,C(c_\bullet).
\]
Подставляя определение \(C(c_\bullet)\) и равенство \(c_m=1/\Delta_m\),
получаем ряд без тригонометрии:
\[
\boxed{
\pi
=
3\sqrt3
-
9\sqrt3
\sum_{m=0}^{\infty}
\frac{1}{\Delta_m^2}
\prod_{j=0}^{m-1}
\frac{2(\Delta_j-2)}{\Delta_j^2}
},
\]
где пустое произведение равно единице.
Если явно подставить радикалы, то то же тождество принимает вид
\[
\resizebox{\textwidth}{!}{%
\(\displaystyle
\pi
=
3\sqrt3
-
9\sqrt3
\sum_{m=0}^{\infty}
\frac{1}{
\left(
\underbrace{2+\sqrt{2+\cdots+\sqrt{2+\sqrt3}}}_{m\ \text{двоек и корней}}
\right)^2
}
\prod_{j=0}^{m-1}
\frac{
2\left(
\underbrace{2+\sqrt{2+\cdots+\sqrt{2+\sqrt3}}}_{j\ \text{двоек и корней}}
-2
\right)
}{
\left(
\underbrace{2+\sqrt{2+\cdots+\sqrt{2+\sqrt3}}}_{j\ \text{двоек и корней}}
\right)^2
}.
\)%
}
\]
Выражение с нулём двоек и корней здесь понимается как \(3\). Первые несколько
слагаемых раскрываются так:
\[
\resizebox{\textwidth}{!}{%
\(\displaystyle
\pi
=
3\sqrt3
-
9\sqrt3\biggl(
\frac19
+
\frac{2}{9(2+\sqrt3)^2}
+
\frac{2}{9\left(2+\sqrt{2+\sqrt3}\right)^2}
\frac{2\sqrt3}{(2+\sqrt3)^2}
+
\frac{2}{9\left(2+\sqrt{2+\sqrt{2+\sqrt3}}\right)^2}
\frac{2\sqrt3}{(2+\sqrt3)^2}
\frac{2\sqrt{2+\sqrt3}}{\left(2+\sqrt{2+\sqrt3}\right)^2}
+
\cdots
\biggr).
\)%
}
\]
Таким образом, \(\pi\) выражается только через арифметические операции и
последовательное извлечение квадратных корней.

Если оборвать ряд после \(M\) слагаемых, ошибка асимптотически равна
\[
\frac{\pi^3}{27}\,4^{-M},
\]
так что каждый новый шаг примерно делит ошибку на четыре.

То же построение работает для произвольного правильного \(n\)-угольника. В этом
случае
\[
N_m=n2^m
\]
и следует выбирать
\[
c_m^{(n)}
=
\frac{1}{
4\cos^2\left(\dfrac{\pi}{n2^{m+1}}\right)
}.
\]
Тогда после \(m\) шагов получается правильный \(n2^m\)-угольник, описанный
около той же окружности.

Для правильного \(n\)-угольника со стороной \(\ell\)
\[
B(P)
=
-n\ell^2\sin\frac{2\pi}{n}.
\]
Поскольку
\[
A(P)
=
\frac{n\ell^2}{4\tan(\pi/n)},
\]
получаем
\[
B(P)
=
-8\sin^2\frac{\pi}{n}\,A(P).
\]
Следовательно,
\[
A_\infty
=
\left(
1-8C(c_\bullet)\sin^2\frac{\pi}{n}
\right)A(P).
\]

Если исходный правильный \(n\)-угольник описан около окружности радиуса \(r\),
то
\[
A(P)=nr^2\tan\frac{\pi}{n},
\qquad
A_\infty=\pi r^2.
\]
Поэтому для соответствующей последовательности коэффициентов
\[
C(c_\bullet^{(n)})
=
\frac{
1-\dfrac{\pi}{n\tan(\pi/n)}
}{
8\sin^2(\pi/n)
}.
\]

В конце отметим, что из этого же примера получается эллиптический вариант.
Схема отсечения углов аффинно инвариантна: если к исходной ломаной применить
невырожденное аффинное преобразование, то на каждом ребре сохраняются те же
параметры деления, а все последующие ломаные и предельная кривая переходят
в свои аффинные образы. Поэтому аффинный образ построенной выше окружности
даёт эллипс, касающийся сторон произвольного невырожденного треугольника
в их серединах. В стандартной терминологии геометрии треугольника этот
конкретный эллипс называется \emph{эллипсом Штейнера}: это единственный эллипс,
вписанный в треугольник и касающийся трёх его сторон в их серединах
\cite{MindaPhelps2008}.

\subsection{Гипербола и алгебраический ряд для \texorpdfstring{\(\log 2\)}{log 2}}

Теперь перейдём от круговой симметрии к более общему принципу. В случае
окружности все локальные конфигурации правильного описанного многоугольника
одинаковы с точностью до поворота. Для гиперболы роль поворотов будет играть
другая группа линейных преобразований. Важно не то, что кривая является
окружностью, а то, что она переходит в себя, её касательные переходят в
касательные, а локальные конфигурации переводятся друг в друга с сохранением
отношений на прямых. Тогда коэффициент отсечения достаточно найти в одной
стандартной конфигурации, после чего тот же коэффициент можно использовать на
всём уровне уточнения.

Классическое изложение теории коник и касательных к ним можно найти,
например, у Салмона~\cite{Salmon1879}; ниже нам понадобятся только
элементарные вычисления для конкретной гиперболы.

Рассмотрим ветвь гиперболы
\[
\mathcal H
=
\{(x,y)\in\mathbb R^2:xy=1,\ x>0,\ y>0\}
\]
между точками
\[
(1,1)
\qquad\text{и}\qquad
\left(2,\frac12\right).
\]
Площадь под этой дугой равна
\[
\int_1^2\frac{dx}{x}=\log 2.
\]
Наша цель --- выбрать коэффициенты отсечения так, чтобы соответствующая
предельная кривая была именно этой дугой гиперболы. Тогда из формулы площади
получится некоторое выражение для логарифма двух.

Удобная параметризация гиперболы имеет вид
\[
H(t)
=
\begin{pmatrix}
e^t\\
e^{-t}
\end{pmatrix}.
\]
Действительно,
\[
H(0)=(1,1),
\qquad
H(\log 2)=\left(2,\frac12\right).
\]

Построение опирается на однопараметрическую группу линейных преобразований
\[
G_s
=
\begin{pmatrix}
e^s&0\\
0&e^{-s}
\end{pmatrix}.
\]
Явно,
\[
G_s
\begin{pmatrix}
x\\
y
\end{pmatrix}
=
\begin{pmatrix}
e^sx\\
e^{-s}y
\end{pmatrix}.
\]

Заметим, что если
\((x,y)\in\mathcal H\), то
\[
(e^sx)(e^{-s}y)=xy=1.
\]
Следовательно,
\[
G_s(\mathcal H)=\mathcal H.
\]
Кроме того,
\[
H(t+s)=G_sH(t).
\]
Значит, сдвиг параметра \(t\mapsto t+s\) реализуется линейным преобразованием
плоскости.

Касательная к гиперболе в точке \(H(t)\) имеет уравнение
\[
\ell_t:\qquad e^{-t}x+e^ty=2.
\]
Преобразования \(G_s\), в силу гладкости, сохраняют всё семейство касательных. 

Такая инвариантность делает возможным единое правило отсечения углов.
Каждая локальная конфигурация
\[
\ell_{t-h},\qquad \ell_t,\qquad \ell_{t+h}
\]
переходит под действием \(G_{-t}\) в стандартную конфигурацию
\[
\ell_{-h},\qquad \ell_0,\qquad \ell_h.
\]
Поскольку аффинные преобразования сохраняют отношения, в которых точки делят
отрезок, коэффициент отсечения зависит только от шага параметра \(h\), а не от
положения \(t\) вдоль гиперболы. Следовательно, на каждом уровне уточнения
можно использовать один общий коэффициент во всех вершинах.

Зафиксируем \(h>0\) и рассмотрим касательные
\[
\ell_{kh},
\qquad k\in\mathbb Z.
\]
Пусть
\[
q_k=\ell_{kh}\cap\ell_{(k+1)h}
\]
--- точки пересечения соседних касательных. Ребро
\[
[q_{k-1},q_k]
\]
лежит на касательной \(\ell_{kh}\).

Сначала покажем, что точка касания \(H(kh)\) является серединой этого ребра.
Применяя \(G_{-kh}\), достаточно проверить утверждение при \(k=0\).

Касательная \(\ell_0\) имеет уравнение
\[
x+y=2.
\]
Параметризуем её так:
\[
(x,y)=(1+u,1-u).
\]
Подстановка в уравнение \(\ell_h\) даёт
\[
e^{-h}(1+u)+e^h(1-u)=2,
\]
откуда
\[
u
=
\frac{\cosh h-1}{\sinh h}
=
\tanh\frac h2.
\]
Поэтому точке \(\ell_0\cap\ell_h\) соответствует
\[
u=\tanh\frac h2,
\]
а точке \(\ell_{-h}\cap\ell_0\) соответствует
\[
u=-\tanh\frac h2.
\]
Их середина соответствует \(u=0\), то есть равна
\[
(1,1)=H(0).
\]
Следовательно, для всякого \(k\)
\[
H(kh)=\frac{q_{k-1}+q_k}{2}.
\]

Теперь уточним касательный многоугольник, разделив шаг параметра пополам. Между
старыми касательными вставим новые касательные с параметрами
\[
\left(k+\frac12\right)h.
\]
На ребре, лежащем на \(\ell_0\), концы старого ребра имеют координаты
\[
u=-a,\qquad u=a,
\qquad
a=\tanh\frac h2,
\]
а две новые точки имеют координаты
\[
u=-b,\qquad u=b,
\qquad
b=\tanh\frac h4.
\]

Пусть \(c\) --- коэффициент отсечения. Левая новая точка должна удовлетворять
\[
-b=(1-c)(-a)+ca.
\]
Отсюда
\[
c=\frac{a-b}{2a}.
\]
Используя формулу двойного аргумента
\[
a=\tanh\frac h2
=
\frac{2b}{1+b^2},
\]
получаем
\[
c
=
\frac{1-b^2}{4}
=
\frac{1}{4\cosh^2(h/4)}.
\]

Легко убедиться, что коэффициент совпадает и для противоположной стороны, поэтому касательный многоугольник с шагом параметра \(h\) переходит в
касательный многоугольник с шагом параметра \(h/2\) при коэффициенте отсечения
\[
c(h)=\frac{1}{4\cosh^2(h/4)}.
\]

Для рассматриваемой дуги начальный шаг параметра равен
\[
h_0=\log 2.
\]
После \(m\) шагов уточнения
\[
h_m=\frac{\log 2}{2^m}.
\]
Поэтому нужные коэффициенты равны
\[
c_m=\frac{1}{4\cosh^2\left(\dfrac{\log 2}{2^{m+2}}\right)}\longrightarrow \frac14.
\]
Значит, коэффициенты отсечения удовлетворяют условиям сходимости из раздела
«Предельная кривая и сходимость площадей», и результаты этого раздела
применимы к построенной последовательности касательных многоугольников.

Теперь можно исключить гиперболические функции. Используем элементарное
тождество
\[
\frac{1}{4\cosh^2 z}
=
\frac{1}{(e^z+e^{-z})^2}
=
\frac{e^{2z}}{(1+e^{2z})^2}.
\]
В нашем случае
\[
z=\frac{\log 2}{2^{m+2}},
\qquad
e^{2z}
=
e^{(\log 2)/2^{m+1}}
=
2^{1/2^{m+1}}
=
\sqrt[2^{m+1}]{2}.
\]
Поэтому, если положить
\[
a_m=\sqrt[2^{m+1}]{2},
\]
то из предыдущего тождества сразу получается
\[
c_m
=
\frac{a_m}{(1+a_m)^2}
=
\frac{\sqrt[2^{m+1}]{2}}
{\left(1+\sqrt[2^{m+1}]{2}\right)^2}.
\]
В частности,
\[
c_0
=
\frac{\sqrt2}{(1+\sqrt2)^2}
=
3\sqrt2-4.
\]

Ту же последовательность можно задать рекуррентно. Из формулы половинного
аргумента для гиперболического косинуса получаем
\[
c_{m+1}
=
\frac{\sqrt{c_m}}{1+2\sqrt{c_m}}.
\]
Итак,
\[
c_0=3\sqrt2-4,
\qquad
c_{m+1}
=
\frac{\sqrt{c_m}}{1+2\sqrt{c_m}},
\]
что задаёт все коэффициенты только при помощи арифметических операций и
квадратных корней.

Для вычисления площади воспользуемся локальной интерпретацией
\(C(c_\bullet)\): рассматриваемая конфигурация получается из стандартного
угла линейным преобразованием, а площадь при таком преобразовании умножается
на его определитель. Пусть
\[
L=\log 2.
\]
Три последовательные вершины начального касательного многоугольника равны
\[
q_{-1}
=
\ell_{-L}\cap\ell_0
=
\left(\frac23,\frac43\right),
\]
\[
q_0
=
\ell_0\cap\ell_L
=
\left(\frac43,\frac23\right),
\]
и
\[
q_1
=
\ell_L\cap\ell_{2L}
=
\left(\frac83,\frac13\right).
\]
Середины двух соседних рёбер равны
\[
\frac{q_{-1}+q_0}{2}=(1,1),
\qquad
\frac{q_0+q_1}{2}
=
\left(2,\frac12\right).
\]

Начальная локальная ломаная между этими двумя точками состоит из двух отрезков
касательных. Первый лежит на прямой
\[
y=2-x,
\]
а второй --- на прямой
\[
y=1-\frac{x}{4}.
\]
Площадь под этой ломаной равна
\[
\int_1^{4/3}(2-x)\,dx
+
\int_{4/3}^{2}
\left(1-\frac{x}{4}\right)\,dx
=
\frac23.
\]
Площадь под предельной дугой гиперболы равна
\[
\int_1^2\frac{dx}{x}
=
\log 2.
\]

Используем замечание~\ref{rem:local-constant-interpretation}:
сравним один локальный кусок между серединами
соседних рёбер: разность площади под предельной дугой и площади под
начальной ломаной равна площади соответствующего локального среза.

Кроме того,
\[
[q_{-1}-q_0,\;q_1-q_0]
=
-\frac23.
\]
Значит, соответствующее линейное преобразование имеет масштаб площади
\(2/3\). С учётом ориентации локальная интерпретация \(C(c_\bullet)\)
даёт
\[
\log 2-\frac23
=
\frac23\,C(c_\bullet).
\]
Следовательно,
\[
\log 2
=
\frac23+\frac23\,C(c_\bullet),
\]
и поэтому
\[
\log 2
=
\frac23
+
\frac13
\sum_{m=0}^{\infty}
c_m^2
\prod_{j=0}^{m-1}
2c_j(1-2c_j).
\]

Наконец, подставим
\[
c_m=\frac{a_m}{(1+a_m)^2},
\qquad
a_m=\sqrt[2^{m+1}]{2}.
\]
Так как
\[
a_{m+1}^2=a_m,
\]
слагаемое упрощается до
\[
c_m^2
\prod_{j=0}^{m-1}
2c_j(1-2c_j)
=
\frac{
3\cdot 2^{m+1}(a_m-1)^3
}{
(a_m+1)(a_m^2+1)
}.
\]

В итоге приходим к чисто алгебраическому тождеству
\[
\boxed{
\log 2
=
\frac23
+
\sum_{m=0}^{\infty}
\frac{
2^{m+1}
\left(\sqrt[2^{m+1}]{2}-1\right)^3
}{
\left(\sqrt[2^{m+1}]{2}+1\right)
\left(\sqrt[2^m]{2}+1\right)
}
}.
\]

\section{Заключение и дальнейшие вопросы}

В работе рассмотрены симметрические нестационарные схемы отсечения углов,
в которых на каждом уровне уточнения один и тот же коэффициент применяется
ко всем вершинам, но сам коэффициент может меняться от шага к шагу.
Для таких схем установлено существование непрерывной предельной кривой
и доказана сходимость ориентированных площадей аппроксимирующих
многоугольников к площади этого предела.

Основной результат состоит в формуле
\[
A_\infty=A(P)+C(c_\bullet)B(P),
\]
где величины \(A(P)\) и \(B(P)\) зависят только от исходного
ориентированного многоугольника, а константа
\[
C(c_\bullet)
=
\sum_{m=0}^{\infty}
\frac{c_m^2}{2}
\prod_{j=0}^{m-1}2c_j(1-2c_j)
\]
зависит только от последовательности коэффициентов отсечения.
Таким образом, геометрические данные начального контура и правило его
последовательного уточнения входят в ответ независимо друг от друга.
Локальная интерпретация этой константы как площади среза стандартного
угла объясняет аффинную природу полученной формулы.

Рассмотренные примеры показывают две стороны этого результата.
В стационарном случае формула непосредственно суммируется и, в частности,
описывает изменение площади для кривых де Рама и схемы Чайкина.
В обратном направлении заранее известная предельная кривая позволяет
вычислять константу \(C(c_\bullet)\) геометрически. Так возникают
представления чисел \(\pi\) и \(\log 2\) в виде рядов, определяемых
последовательными отсечениями. Поэтому процедура отсечения углов может
рассматриваться не только как способ построения кривых, но и как
геометрический механизм разложения площади на последовательные локальные
вклады.

Естественным продолжением работы было бы исследование регулярности
предельной кривой для общих нестационарных последовательностей
\(c_\bullet\). Интерес представляет также обратная задача: описать
последовательности коэффициентов, для которых предел совпадает с заранее
заданной кривой. В частности, возникает вопрос о классификации кривых,
семейство касательных которых совместимо с единым коэффициентом отсечения
на каждом уровне. Другими возможными направлениями являются схемы с
коэффициентами, зависящими от вершины, и многомерные аналоги формулы
площади для усечения многогранников.

\section*{Благодарности}
\addcontentsline{toc}{section}{Благодарности}

Автор благодарит Н.~Н.~Косовского за подсказанную тему работы,
Д.~В.~Карпова за помощь с терминологией на раннем этапе и
А.~Д.~Баранова за консультации по тонкостям сходимости.

\renewcommand{\refname}{Список литературы}
\phantomsection